\documentclass[a4wide]{amsart}
\usepackage{a4wide, amssymb,amsmath}
\usepackage{euscript,enumerate}
\usepackage{color}
\usepackage{kotex}
\usepackage{hyperref, cleveref}
\usepackage{mathtools,xparse,esint}

\theoremstyle{plain}
\newtheorem{theorem}{Theorem}[section]
\newtheorem{lemma}[theorem]{Lemma}
\newtheorem{corollary}[theorem]{Corollary}

\theoremstyle{definition}
\newtheorem{definition}[theorem]{Definition}
\newtheorem{remark}[theorem]{Remark}

\numberwithin{equation}{section}

\makeatletter
\@namedef{subjclassname@2020}{%
	\textup{2020} Mathematics Subject Classification}
\makeatother

\author{Ki-Ahm Lee}
\address{Department of Mathematical Sciences and Research Institute of Mathematics,
	Seoul National University, Seoul 08826, Korea.}
\email{kiahm@snu.ac.kr}

\author{Se-Chan Lee}
\address{Research Institute of Mathematics,
	Seoul National University, Seoul 08826, Korea.}
\email{dltpcks1@snu.ac.kr}

\thanks{ }

\begin{document}

\title{The Wiener criterion for fully nonlinear elliptic equations}


\begin{abstract}
We study the boundary continuity of solutions to fully nonlinear elliptic equations. We first define a capacity for operators in non-divergence form and derive several capacitary estimates. Secondly, we formulate the Wiener criterion, which characterizes a regular boundary point via potential theory. Our approach utilizes the asymptotic behavior of homogeneous solutions, together with Harnack inequality and the comparison principle.
\end{abstract}

\maketitle

\section{Introduction}
Let $\Omega$ be an open and bounded subset in $\mathbb{R}^n$, $f$ be a boundary data on $\partial \Omega$, and $\mathcal{M}$ be an elliptic operator. For the existence of a solution $u$ (in a suitable sense) to the Dirichlet problem
	\begin{align*}
	\left\{ \begin{array}{ll} 
		\mathcal{M}[u]=0 & \textrm{in $\Omega$}, \\
		u=f & \textrm{on $\partial \Omega$},
	\end{array} \right.
	\end{align*}
one may apply Perron's method. If the solvability of the Dirichlet problem on any balls is known and $\mathcal{M}$ allows a comparison principle, it is rather straightforward to prove that the upper Perron solution $\overline{H}_f$ satisfies $\mathcal{M}[\overline{H}_f]=0$ in $\Omega$. (See Section 2 for details.) Nevertheless, we cannot ensure that the boundary condition $u=f$ on $\partial \Omega$ is satisfied by the upper Perron solution, in general. Instead, we are forced to discover an additional condition for the boundary $\partial \Omega$, which enables us to capture the boundary behavior of $\overline{H}_f$. 

To be precise, we say a boundary point $x_0 \in \partial \Omega$ is \textit{regular} with respect to $\Omega$, if 
\begin{align*}
\lim_{\Omega \ni y \to x_0}\overline{H}_f(y)=f(x_0).
\end{align*}
whenever $f \in C(\partial \Omega)$. One simple characterization of a regular boundary point is to find a \textit{barrier function}; see Section 2 for the precise definition. As a consequence, by constructing proper barrier functions, geometric criteria on $\partial \Omega$ such as an exterior sphere condition or an exterior cone condition have been invoked to guarantee the boundary continuity at $x_0 \in \partial \Omega$ for a variety of elliptic operators. 

On the other hand, Wiener \cite{Wie24a} developed an alternative criterion for a regular boundary point, based on potential theory. Namely, for the Laplacian operator ($\mathcal{M}=\Delta$), $x_0 \in \partial \Omega$ is regular if and only if
the Wiener integral diverges, i.e.
\begin{align*}
\int_0^1 \frac{\mathrm{cap}_2(\overline{B_t(x_0)} \setminus \Omega, B_{2t}(x_0))}{\mathrm{cap}_2(\overline{B_t(x_0)}, B_{2t}(x_0))} \frac{\mathrm{d}t}{t}=\infty,
\end{align*}
where $\mathrm{cap}_2(K, \Omega)$ is defined by the variational capacity of the Laplacian operator. Surprisingly, the Wiener criterion becomes both a sufficient and necessary condition for the regularity of a boundary point. Here the notion of capacity is used to measure the `size' of sets in view of given differential equations. Roughly speaking, $x_0 \in \partial \Omega$ is regular if and only if 
$\Omega^c$ is `thick' enough at $x_0$ in the potential theoretic sense. 

Both linear and nonlinear potential theory have been extensively studied in literature; see \cite{Bre60, HKM93, Hel09, Lan72, MZ97, TW02} and references therein.
Since the main ingredient of potential theory comes from the integration by parts, the theory and corresponding Wiener criterion have been developed mostly for operators in divergence form. Littman, Stampacchia and Weinburger \cite{LSW63} demonstrated the coincidence between the regular points for uniformly elliptic operators $\mathcal{M}=-D_j(a_{ij}D_i)$, where $a_{ij}$ is bounded and measurable, and for the Laplacian operator. For the $p$-Laplacian operator ($\mathcal{M}=\Delta_p, p>1$), Maz'ya \cite{Maz70} verified the sufficiency of the $p$-Wiener criterion, i.e. $x_0 \in \partial \Omega$ is regular for $\Delta_p$ if 
\begin{align*}
\int_0^1 \Big(\frac{\mathrm{cap}_p(\overline{B_t(x_0)} \setminus \Omega, B_{2t}(x_0))}{\mathrm{cap}_p(\overline{B_t(x_0)}, B_{2t}(x_0))}\Big)^{1/(p-1)} \frac{\mathrm{d}t}{t}=\infty.
\end{align*} 
For the converse direction, Lindqvist and Martio \cite{LM85} proved the necessity of the Wiener criterion under the assumption $p>n-1$. Later, Kilpel\"ainen and Mal\'y \cite{KM94} extended this result to any $p>1$, via the Wolff potential estimate. For the other available results on the Wiener criterion, we refer to \cite{AK04} for $p(x)$-Laplacian operators and \cite{LL21} for operators with Orlicz growth. Note that all of these results consider elliptic operators in divergence form.

For elliptic operators in non-divergence form, relatively small amounts of results for the Wiener criterion are known. While the equivalence was obtained for $\mathcal{M}=D_j(a_{ij}D_i)$ with merely measurable coefficients in \cite{LSW63}, Miller \cite{Mil68, Mil70} discovered the non-equivalence with respect to $\mathcal{M}=a_{ij}D_{ij}u$, even if the coefficients $a_{ij}$ are continuous. More precisely, he presented examples of linear operators $\mathcal{M}$ in non-divergence form and domains $\Omega$ such that $x_0 \in \partial \Omega$ is regular for $\mathcal{M}$, but $x_0$ is irregular for $\Delta$, and vice-versa. We also refer \cite{KY17, Leb13}. On the other hand, Bauman \cite{Bau85} developed the Wiener test for $\mathcal{M}=a_{ij}D_{ij}u$ with continuous coefficients $a_{ij}$. He proved that $x_0 \in \partial \Omega$ is regular if and only if
\begin{enumerate}[(i)]
	\item $\mathrm{cap}_{\mathcal{M}}(\{x_0\})>0$, or
	\item $\sum_{j=1}^{\infty}\widetilde{g}(x_0, x_0+2^{-j}e) \cdot \mathrm{cap}_{\mathcal{M}}(\Omega^c \cap (\overline{B_{2^{-j}}(x_0)} \setminus B_{2^{-j-1}}(x_0)))=\infty$.
\end{enumerate} 
Here $\widetilde{g}$ is the normalized Green function and $e$ is a unit vector in $\mathbb{R}^n$.

The goal of this paper is to establish the Wiener criterion for fully nonlinear elliptic operators, by implementing potential theoretic tools.
To illustrate the issues, we consider an \textit{Issacs operator}, i.e. an operator $F$ with the following two properties:
\begin{enumerate}[]
	\item (F1) $F$ is \textit{uniformly elliptic}: there exist positive constants $0<\lambda \leq \Lambda$ such that for any $M \in \mathcal{S}^n$,
	\begin{align*}
	\lambda \|N\| \leq F(M+N)-F(M) \leq \Lambda \|N\|, \quad \forall N \geq 0.
	\end{align*}
	Here we write $N\geq0$ whenever $N$ is a non-negative definite symmetric matrix.
	\item  (F2) $F$ is \textit{positively homogeneous of degree one}:  $F(tM)=tF(M)$ for any $t>0$ and $M \in \mathcal{S}^n$.
\end{enumerate}

Throughout the present paper, we suppose that $F$ satisfies (F1) and (F2), unless otherwise stated. Typical examples of operators satisfying (F1) and (F2) are the Pucci extremal operators $\mathcal{P}_{\lambda, \Lambda}^+$ and $\mathcal{P}_{\lambda, \Lambda}^-$, defined by
\begin{align*}
\mathcal{P}_{\lambda, \Lambda}^+(M)=\Lambda \sum_{e_i>0}e_i+\lambda \sum_{e_i<0}e_i, \quad \mathcal{P}_{\lambda, \Lambda}^-(M)=\Lambda \sum_{e_i<0}e_i+\lambda \sum_{e_i>0}e_i,
\end{align*}
where $e_i=e_i(M)$ are the eigenvalues of $M$. For a fully nonlinear operator $F$ satisfying (F1) and (F2), we define a \textit{dual operator} 
\begin{align*}
\widetilde{F}(M):=-F(-M), \quad \text{for $M \in \mathcal{S}^n$}.
\end{align*}
Then it is obvious that $\widetilde{F}$ also satisfies (F1) and (F2). One important property that $F$ satisfying (F1) and (F2) possesses is the existence of a homogeneous solution $V$:
\begin{lemma}[A homogeneous solution; \cite{ASS11, CL08}] \label{fund}
	There exists a non-constant solution of $F(D^2u)=0$ in $\mathbb{R}^n\setminus\{0\}$ that is bounded below in $B_1$ and bounded above in $\mathbb{R}^n\setminus B_1$. Moreover, the set of all such solutions is of the form $\{aV+b \, | \, a>0, b\in \mathbb{R}\}$, where 
	$V \in C_{\mathrm{loc}}^{1, \gamma}(\mathbb{R}^n \setminus \{0\})$ can be chosen to satisfy one of the following homogeneity relations: for all $t>0$
	\begin{align*}
		V(x)=V(tx)+\log t \quad \textrm{in $\mathbb{R}^n \setminus\{0\}$} \quad \textrm{where $\alpha^{\ast}=0$,}
	\end{align*}
	or
	\begin{align*}
		V(x)=t^{\alpha^{\ast}}V(tx), \ \alpha^{\ast}V>0 \quad \textrm{in $\mathbb{R}^n \setminus\{0\}$},
	\end{align*}
	for some number $\alpha^{\ast} \in (-1, \infty) \setminus\{0\}$ that depends only on $F$ and $n$. We call the number $\alpha^{\ast}=\alpha^{\ast}(F)$ the \textit{scaling exponent} of $F$.
\end{lemma}
Now we are ready to state our first main theorem, namely, the sufficiency of the Wiener criterion:
\begin{theorem}[The sufficiency of the Wiener crietrion]\label{wiener}
	If
	\begin{align*}
	\int_0^{1} {\mathrm{cap}_{{F}} (\overline{B_t(x_0)} \setminus \Omega, B_{2t}(x_0))}\frac{\mathrm{d}t}{t}=\infty
	\end{align*}
	and
	\begin{align*}
	\int_0^{1} {\mathrm{cap}_{\widetilde{F}} (\overline{B_t(x_0)} \setminus \Omega, B_{2t}(x_0))}\frac{\mathrm{d}t}{t}=\infty,
	\end{align*}
	then the boundary point $x_0 \in \partial \Omega$ is ($F$-)regular.
\end{theorem}
We remark that the Wiener integral is again defined in terms of a capacity, but the definition of a $F$-capacity is quite different from the variational capacity for the Laplacian case; see Section 3 for details. Furthermore, as a corollary of Theorem \ref{wiener}, we will derive the quantitative estimate for a modulus of continuity at a regular boundary point (Lemma \ref{modulus}), and suggest another geometric condition, called an exterior corkscrew condition (Corollary \ref{corkscrew}).

Our second main theorem is concerned with the necessity of the Wiener criterion. We propose a partial result on the necessary condition, i.e. exploiting the additional structure of $F$, we show that the Wiener integral at $x_0 \in \partial \Omega$ must diverge whenever $x_0$ is a regular boundary point.
\begin{theorem}[The necessity of the Wiener criterion]\label{wienernec}
	Suppose that $F$ is concave and $\alpha^{\ast}(F)<1$.
	If a boundary point $x_0 \in \partial \Omega$ is regular, then 
	\begin{align*}
	\int_0^{1} {\mathrm{cap}_{{F}} (\overline{B_t(x_0)} \setminus \Omega, B_{2t}(x_0))}\frac{\mathrm{d}t}{t}=\infty.
	\end{align*}
\end{theorem}
Note that the assumption $\alpha^{\ast}(F)<1$ in the fully nonlinear case corresponds to the assumption $p>n-1$ in the $p$-Laplacian case, \cite{LM85}. The underlying idea for both cases is to utilize the non-zero capacity of a line segment (or a set of Hausdorff dimension 1). Further comments on
this assumption can be found in Section 5.

In this paper, the main difficulty arises from the inherent lack of divergence structure; we cannot define a variational capacity by means of an energy minimizer, and moreover, we cannot employ integral estimates involving Sobolev inequality and Poincar\'e inequality. Instead, we will develop potential theory with non-divergence structure by the construction of appropriate barrier functions using the homogeneous solution, and by the application of the comparison principle and Harnack inequality. In short, our strategy is to capture the local boundary behavior of the upper Perron solution $\overline{H}_f$ in terms of newly defined capacity $\mathrm{cap}_F(K, B)$ and the capacity potential (or the balayage) $\hat{R}_K^1(B)$, using prescribed tools. Heuristically, the non-variational capacity measures the `height' of the $F$-solution with the boundary value $0$ on $\partial B$ and $1$ on $\partial K$, while the variational capacity measures the `energy' of such function. We emphasize that although our notion of capacity does not satisfy the subadditive property in general, it was still able to recover certain properties of the variational capacity.

Finally, we would like to point out that the dual operator $\widetilde{F}$ is different from $F$, for general $F$. Thus, even though $u$ is an $F$-supersolution, we cannot guarantee $-u$ is an $F$-subsolution. Moreover, a similar feature is found in the growth rate of the homogeneous solution for $F$; two growth rates of an upward-pointing homogeneous solution and a downward-pointing one can be different. This phenomenon naturally leads us 
\begin{enumerate}[(i)]
	\item to describe the local behavior of both the upper Perron solution $\overline{H}_f$ and the lower Perron solution $\underline{H}_f$ for regularity at $x_0 \in \partial \Omega$;
	\item to construct two (upper/lower) barrier functions when characterizing a regular boundary point;
	\item to display two different Wiener integrals in our main theorem,
\end{enumerate}
which differ from the previous results that appeared in \cite{Bau85, KM94, Wie24a}.

\vspace{0.5cm}
\textbf{Outline.} This paper is organized as follows. In Section 2, we summarize the terminology and preliminary results for our main theorems. In short, we introduce $F$-superharmonic functions and Poisson modification and then perform Perron's method. In Section 3, we first define a balayage and a capacity for uniformly elliptic operators in non-divergence form. Then we prove several capacitary estimates by constructing auxiliary functions and provide the characterization of a regular boundary point via balayage. Section 4 consists of potential theoretic estimates for the capacity potential. Then we prove the sufficiency of the Wiener criterion and several corollaries. Finally, Section 5 is devoted to the proof of the (partial) necessity of the Wiener criterion.

\section{Perron's method}
\subsection{$F$-supersolutions and $F$-superharmonic functions}
In this subsection, we only require the condition (F1) for an operator $F$. To illustrate Perron's method precisely, we start with two different notions of solutions for a uniformly elliptic operator $F$: $F$-solutions and $F$-harmonic functions. Indeed, we will prove that these two notions coincide.

\begin{definition}[$F$-supersolution]
	A lower semi-continuous [resp. upper semi-continuous] function $u$ in $\Omega$ is a \textit{(viscosity) $F$-supersolution} [resp. \textit{(viscosity) $F$-subsolution}] in $\Omega$, when the following condition holds:
	
	if $x_0 \in \Omega$, $\varphi \in C^2(\Omega)$ and $u-\varphi$ has a local minimum at $x_0$, then 
	\begin{align*}
	F(D^2\varphi(x_0)) \leq 0.
	\end{align*} 
	[resp. if $u-\varphi$ has a local maximum at $x_0$, then $F(D^2\varphi(x_0)) \geq 0$.]
	
	 We say that $u \in C({\Omega})$ a \textit{(viscosity) $F$-solution} if $u$ is both an $F$-subsolution and an $F$-supersolution.  
\end{definition}

\begin{lemma}\label{lsc}
	Suppose that a lower semi-continuous function $u$ is an $F$-supersolution in $\Omega$. Then 
	\begin{align*}
	u(x)=\liminf_{\Omega \ni y \to x}u(y) \quad \text{for any $x \in \Omega$}.
	\end{align*}
\end{lemma}

\begin{proof}
	We argue by contradiction: suppose that 
	\begin{align*}
	u(x_0) < \liminf_{\Omega \ni y \to x}u(y) \quad \text{for some $x_0 \in \Omega$.}
	\end{align*}
	Then for any $\varphi \in C^2(\Omega)$, it follows that $u-\varphi$ has a local minimum at $x_0$ and so we can test this function. Therefore, 
	\begin{align*}
	F(D^2\varphi(x_0)) \leq 0 \quad \text{for any $\varphi \in C^2(\Omega)$},
	\end{align*}
	which is impossible. 
\end{proof}

\begin{theorem}\label{stability}
	\begin{enumerate}[(i)]
		\item  (Stability) Let $\{u_k\}_{k \geq 1} \subset C(\Omega)$ be a sequence of $F$-solutions in $\Omega$. Assume that $u_k$ converges uniformly in every compact set of $\Omega$ to $u$. Then $u$ is an $F$-solution in $\Omega$.
		
		\item  (Compactness) Suppose that $\{u_k\}_{k \geq 1} \subset C(\Omega)$ is a locally uniformly bounded sequence of $F$-solutions in $\Omega$. Then it has a subsequence that converges locally uniformly in $\Omega$ to an $F$-solution.
	\end{enumerate}
\end{theorem}

\begin{theorem}[Harnack convergence theorem] \label{harnackct}
	Let $\{u_k\}_{k \geq 1} \subset C(\Omega)$ be an increasing sequence of $F$-solutions in $\Omega$. Then the function $u=\lim_{k \to \infty}u_k$ is either an $F$-solution or identically $+\infty$ in $\Omega$.
\end{theorem}

\begin{proof}
	If $u(x) < \infty$ for some $x \in \Omega$, it follows from Harnack inequality that $u$ is locally bounded in $\Omega$. The interior $C^{\alpha}$-estimate yields that the sequence $u_k$ is equicontinuous in every compact subset of $\Omega$. Thus, applying Arzela-Ascoli theorem and Theorem \ref{stability} (i), we finish the proof.
\end{proof}

We demonstrate two essential tools for Perron's method, namely, the comparison principle and the solvability of the Dirichlet problem in a ball.
\begin{theorem}[Comparison principle for $F$-super/subsolutions, \cite{KK00, KK07}] \label{cpsupersol}
	Let $\Omega$ be a bounded open subset of $\mathbb{R}^n$. Let $v \in \mathrm{USC}(\overline{\Omega})$ [resp.$u \in \mathrm{LSC}(\overline{\Omega})$] be an $F$-subsolution [resp. $F$-supersolution] in $\Omega$ and $v \leq u$ on $\partial \Omega$. Then $v \leq u$ in $\overline{\Omega}$.
\end{theorem}
In the previous theorem, $\mathrm{USC}(\overline{\Omega})$ denotes the set of all upper semi-continuous functions from $\overline{\Omega}$ to $\mathbb{R}$.
Moreover, note that for a lower semi-continuous function $f$, there exists an increasing sequence of continuous functions $\{f_n\}$ such that $f_n \to f$ pointwise as $n \to \infty$.

\begin{theorem}[The solvability of the Dirichlet problem]\label{soldp}
	Let $\Omega$ satisfy a uniform exterior cone condition and $f \in C(\partial \Omega)$. Then there exists a unique $F$-solution $u \in C(\overline{\Omega})$ of the Dirichlet problem
	\begin{align*}
	\left\{ \begin{array}{ll} 
	F(D^2u)=0 & \textrm{in $\Omega$}, \\
	u=f & \textrm{on $\partial \Omega$}.
	\end{array} \right.
	\end{align*}
\end{theorem}

\begin{proof}
	The existence depends on the construction of global barriers achieving given boundary data and the standard Perron's method; see \cite{CKLS99, Mic81} and \cite{CIL92, IL90}. Then the uniqueness comes from the comparison principle, Theorem \ref{cpsupersol}.
\end{proof}

\begin{definition}[$F$-superharmonic function]
	A function $u :\Omega \to (-\infty, \infty]$ is called \textit{$F$-superharmonic} if 
	\begin{enumerate}[(i)]
		\item $u$ is lower semi-continuous;
		\item $u \not\equiv \infty$ in each component of $\Omega$;
		\item $u$ satisfies the comparison principle in each open $D \subset \subset \Omega$: If $h \in C(\overline{D})$ is an $F$-solution in $D$, and if $h \leq u$ on $\partial D$, then $h \leq u$ in $D$.
	\end{enumerate}
Analagously, a function $u :\Omega \to [-\infty, \infty)$ is called \textit{$F$-subharmonic} if 
\begin{enumerate}[(i)]
	\item $u$ is upper semi-continuous;
	\item $u \not\equiv -\infty$ in each component of $\Omega$;
	\item $u$ satisfies the comparsion principle in each open $D \subset \subset \Omega$: If $h \in C(\overline{D})$ is an $F$-solution in $D$, and if $h \geq u$ on $\partial D$, then $h \geq u$ in $D$.
\end{enumerate}
 We say that $u \in C({\Omega})$ is \textit{$F$-harmonic} if $u$ is both $F$-subharmonic and $F$-superharmonic.  
\end{definition}

\begin{lemma}\label{easy}
	\begin{enumerate}[(i)]
		\item If $u$ is $F$-superharmonic, then $au+b$ is $F$-superharmonic whenever $a$ and $b$ are real numbers and $a \geq 0$.
		
		\item If $u$ and $v$ are $F$-superhmaronic, then the function $\min\{u,v\}$ is $F$-superharmonic.
		
		\item Suppose that $u_i$, $i=1,2, \cdots$, are $F$-superharmonic in $\Omega$. If the sequence $u_i$ is increasing or converges uniformly on compact subsets of $\Omega$, then in each component of $\Omega$, the limit function $u=\lim_{i \to \infty} u_i$ is $F$-superharmonic unless $u \equiv \infty$.
	\end{enumerate}
\end{lemma}

%

\begin{theorem}[Comparison principle for $F$-super/subharmonic functions] \label{cpsuperhar}
	Suppose that $u$ is $F$-superharmonic and that $v$ is $F$-subharmonic in $\Omega$. If 
	\begin{align*}
	\limsup_{y \to x}v(y) \leq \liminf_{y \to x}u(y)
	\end{align*}
	for all $x \in \partial \Omega$, then $v \leq u$ in $\Omega$.
\end{theorem}

\begin{proof}
	Fix $\varepsilon>0$ and let 
	\begin{align*}
		K_{\varepsilon}:=\{x\in \Omega: v(x) \geq u(x)+\varepsilon\}.
	\end{align*}
	Then $K_{\varepsilon}$ is a compact subset of $\Omega$ and so there exists an open cover $D_{\varepsilon}$ such that $K_{\varepsilon} \subset D_{\varepsilon} \subset \Omega$ where $D_{\varepsilon}$ is a union of finitely many balls $B_i$, and $\partial D_{\varepsilon} \subset \Omega \setminus K_{\varepsilon}$. Since $u$ is lower semi-continuous, $v$ is upper semi-continuous and $\partial D_{\varepsilon}$ is compact, we can choose a continuous function $\theta$ on $\partial D_{\varepsilon}$ such that $v \leq \theta \leq u+\varepsilon$ on $\partial D_{\varepsilon}$. Moreover, since $D_{\varepsilon}$ satisfies a uniform exterior cone condition, there exists $h \in C(\overline{D})$ which is the unique $F$-solution in $D_{\varepsilon}$ that coincides with $\theta$ on $\partial D_{\varepsilon}$ by applying Theorem \ref{soldp}. Now the definition of $F$-super/subharmonic functions yields that 
	\begin{align*}
	v \leq h \leq  u+\varepsilon \quad \text{in $D_{\varepsilon}$}.
	\end{align*}
	Hence, $v \leq u+\varepsilon$ in $\Omega$ and the desired result follows by letting $\varepsilon \to 0$.
\end{proof}

Now we describe the equivalence of $F$-supersolution and $F$-superharmonic function; see also \cite{Hor94, KKL19, Lab02b}. 
\begin{theorem}\label{equiv}
	$u$ is an $F$-supersolution in $\Omega$ if and only if $u$ is $F$-superharmonic in $\Omega$.
\end{theorem}
\begin{proof}
	Assume first that $u$ is an $F$-supersolution in $\Omega$. To show that $u$ is $F$-superharmonic, we only need to verify the property (iii) in the definition of $F$-superharmonic functions. Let $D \subset \subset \Omega$ be an open set and take $h \in C(\overline{D})$ to be an $F$-solution in $D$ such that $h \leq u$ on $\partial D$. Thus, applying the comparison principle for $F$-super/subsolutions (Theorem \ref{cpsupersol}) for $u$ and $h$, we conclude that $h \leq u$ in $\overline{D}$.
	
	Assume now that $u$ is $F$-superharmonic in $\Omega$. For any $x_0 \in \Omega$, take $\varphi \in C^2(B_r(x_0))$ such that $u-\varphi$ has a local minimum $0$ at $x_0$. We need to prove that
	\begin{align} \label{testsuper}
	F(D^2\varphi(x_0)) \leq 0.
	\end{align}
	We argue by contradiction; suppose that \eqref{testsuper} fails. By the continuity of the operator $F$, there exist $\tau >0$ and $\rho \in (0,r)$ such that 
	\begin{align*}
	F(D^2\varphi(x)) > \tau \quad \text{in $B_{\rho}(x_0)$}.
	\end{align*}
	Consider a cut-off function $\eta \in C_0^2(B_{\rho}(x_0))$ with $\mathrm{supp}\,\eta \subset B_{\rho/2}(x_0)$ and $\eta(x_0)=1$. Since the uniform ellipticity gives
	\begin{align*}
	F(D^2(\varphi+\varepsilon\eta)) \geq F(D^2\varphi)+\varepsilon\, \mathcal{P}^-_{\lambda, \Lambda}(D^2\eta) \quad \text{for any $\varepsilon>0$},
	\end{align*}
	we can choose a sufficiently small $\varepsilon_0>0$ so that
	\begin{align*}
	F(D^2(\varphi+\varepsilon_0\eta)) \geq 0 \quad \text{in $B_{\rho}(x_0)$.}
	\end{align*} 
	In other words, since $\varphi+\varepsilon_0\eta \in C^2(B_{\rho}(x_0))$, $\varphi+\varepsilon_0 \eta$ is an $F$-subsolution in $B_{\rho}(x_0)$. Furthermore, by a similar argument as in the first part, we have $\varphi+\varepsilon_0 \eta$ is $F$-subharmonic in $B_{\rho}(x_0)$. On the other hand, on $\partial B_{\rho/2}(x_0)$, we have
	\begin{align*}
	\varphi(x)+\varepsilon_0\eta(x)=\varphi(x) \leq u(x).
	\end{align*}
	Thus, by the comparison principle for $F$-super/subharmonic functions (Theorem \ref{cpsuperhar}) for $u$ and $\varphi+\varepsilon_0 \eta$, we conclude that $\varphi+\varepsilon_0\eta \leq u$ in $B_{\rho/2}(x_0)$. In particular, letting $x=x_0$, we have $\varphi(x_0)+\varepsilon_0 \leq u(x_0)$, which contradicts to the fact that $u(x_0)=\varphi(x_0)$.
\end{proof}

The result for $F$-subsolution and $F$-subharmonic function can be derived in the same manner and consequently, a function $u$ is an $F$-solution if and only if it is $F$-harmonic.

\subsection{Perron's method}
\begin{lemma}[Pasting lemma] \label{pasting}
	Let $D \subset \Omega$ be open. Also let $u$ and $v$ be $F$-superharmonic in $\Omega$ and $D$, resepctively. If the function
	\begin{align*}
	s:=
	\left\{ \begin{array}{ll} 
	\min\{u, v\} & \textrm{in $D$}, \\
	u & \textrm{in $\Omega \setminus D$},
	\end{array} \right.
	\end{align*}
	is lower semi-continuous, then $s$ is $F$-superharmonic in $\Omega$.
\end{lemma}

\begin{proof}
	Let $G \subset \subset \Omega$ be open and $h \in C(\overline{G})$ be $F$-harmonic such that $h \leq s$ on $\partial G$. Then $h \leq u$ in $\overline{G}$. In particular, since $s$ is lower semi-continuous,
	\begin{align*}
	\lim_{D\cap G \ni y \to x} h(y) \leq u(x)=s(x) \leq \liminf_{D\cap G \ni y \to x} v(y)
	\end{align*}
	for all $x \in \partial D \cap G$. Thus,
	\begin{align*}
	\lim_{D\cap G \ni y \to x} h(y) \leq s(x) \leq \liminf_{D\cap G \ni y \to x}v(y)
	\end{align*}
	for all $x \in \partial (D \cap G)$, and Theorem \ref{cpsuperhar} implies $h \leq v$ in $D \cap G$. Therefore, $h \leq s$ in $G$ and the lemma is proved.
\end{proof}

Suppose that $u$ is $F$-superharmonic in $\Omega$ and that $B \subset\subset \Omega$ is an open ball. Let 
\begin{align*}
u_B(z):=\inf\{v(z) :\text{$v$ is $F$-superharmonic in $B$, $\liminf_{y \to x}v(y) \geq u(x)$ for each $x\in \partial B$}\}.
\end{align*}
Then define the \textit{Poisson modification} $P(u, B)$ of $u$ in $B$ to be the function
\begin{align*}
P(u, B):=
\left\{ \begin{array}{ll} 
u_B & \textrm{in $B$}, \\
u & \textrm{in $\Omega \setminus B$}.
\end{array} \right.
\end{align*}

\begin{lemma}[Poisson modification] \label{poisson}
	The Poisson modification $P(u, B)$ is $F$-superharmonic in $\Omega$, $F$-harmonic in $B$, and $P(u, B) \leq u$ in $\Omega$.
\end{lemma}

\begin{proof}
	By definition, it is clear that $P(u, B) \leq u$ in $\Omega$. To show $P(u, B)$ is $F$-harmonic in $B$, choose an increasing sequence of continuous functions $\{\theta_j\}_{j \geq 1}$ on $\partial B$ such that $u=\lim_{j \to \infty} \theta_j$. (recall that this is possible since $u$ is lower semi-continuous.) Then let $h_j \in C(\overline{B})$ be the $F$-solution of the Dirichlet problem $F(D^2h_j)=0$ in $B$ and $h_j=\theta_j$ on $\partial B$ by Theorem \ref{soldp}. The comparison principle yields that $h_j$ is also an increasing sequence. Thus, applying Harnack convergence theorem (Theorem \ref{harnackct}), we have the limit function $h=\lim_{j \to \infty}h_j$ is an $F$-solution in $B$. Since
	\begin{align} \label{check}
	\liminf_{y \to x}h(y) \geq \lim_{j \to \infty}\liminf_{y \to x}h_j(y)=\lim_{j \to \infty}h_j(x)=\lim_{j \to \infty}\theta_j(x)=u(x),
	\end{align}
	for any $x \in \partial B$, we have $h \geq P(u, B)$ in $B$ by the definition of $u_B$. On the other hand, since $h_j(x) \leq \liminf_{y \to x}v(y)$ where $x \in \partial B$ and $v$ is an admissible function for $u_B$, we have $h \leq P(u, B)$ in $B$ by applying the comparison principle, letting $j \to \infty$ and taking the infimum over $v$. Therefore, $P(u, B)=h$ is $F$-harmonic in $B$.
	
	Finally, if we show that $P(u, B)$ is lower semi-continuous, then it immediately follows from the pasting lemma that $P(u, B)$ is $F$-superharmonic in $\Omega$. Indeed, it is enough to show that $P(u, B)$ is lower semi-continuous at each point $x \in \partial B$; recall \eqref{check}.
\end{proof}

	\begin{remark}[Perron's method]
	Let $\Omega$ be an open, bounded subset of $\mathbb{R}^{n}$ and $f$ be a bounded function on $\partial \Omega$. The \textit{upper class} $\mathcal{U}_f=\mathcal{U}_f(\Omega)$ consists of all functions $u$ in $\Omega$ such that
	\begin{enumerate}[(i)]
		\item $u$ is $F$-superharmonic in $\Omega$;
		\item $u$ is bounded below;
		\item $
		\liminf_{\Omega \ni y \to x}u(y) \geq f(x) \ \text{for each $x \in \partial \Omega.$}
		$
	\end{enumerate}
	Then, we define the \textit{upper Perron solution} of $f$ by 
	\begin{align*}
	\overline{H}_f=\overline{H}_f(\Omega):=\inf_{u \in \mathcal{U}_f} u.
	\end{align*}
	Similarly, let the \textit{lower class} $\mathcal{L}_f=\mathcal{L}_f(\Omega)$ be the set of all $F$-subharmonic functions $v$ in $\Omega$ which are bounded above and such that 
	\begin{align*}
	\limsup_{\Omega \ni y \to x}v(y) \leq f(x) \quad \text{for each $x \in \partial \Omega,$}
	\end{align*}
	and define the \textit{lower Perron solution} of $f$ by  
	\begin{align*}
	\underline{H}_f=\underline{H}_f(\Omega):=\sup_{v \in \mathcal{L}_f} v.
	\end{align*}
	Then, the comparison principle yields that $\underline{H}_f \leq \overline{H}_f$.
\end{remark}

\begin{lemma}
	The Perron solutions $\overline{H}_f$ and $\underline{H}_f$ are $F$-solutions in $\Omega$.
\end{lemma}

\begin{proof}
	This proof is based on the argument used in \cite{KL96}. Fix an open ball $B$ with $B \subset\subset \Omega$. Next, choose a countable, dense subset 
	$X=\{x_1, x_2, ...\}$ of $B$ and then for each $j=1,2,...$, choose $u_{i,j} \in \mathcal{U}_f$ such that 
	\begin{align*}
	\lim_{i \to \infty}u_{i,j}(x_j)=\overline{H}_f(x_j).
	\end{align*} 
	Moreover, replacing $u_{i,j+1}$ by $\min\{u_{i,j}, u_{i, j+1}\}$ if necessary, we have
	\begin{align*}
	\lim_{i \to \infty} u_{i,j}(x_k)=\overline{H}_f(x_k),
	\end{align*}
	for each $k=1,2...,j$ and each $j$. Now, let $U_{i,j}:=P(u_{i,j}, B)$ be the Poisson modification of $u_{i, j}$ in $B$. Then we observe that $\overline{H}_f \leq U_{i,j} \leq u_{i,j}$ and $U_{i,j}$ is $F$-harmonic in $B$. By compactness (Theorem \ref{stability} (ii)), $U_{i,j}$ converges locally uniformly to $F$-harmonic $v_j$ in $B$ (passing to a subsequence, if necessary). Again by compactness, $v_j$ converges locally uniformly to $F$-harmonic $h$ in $B$.
	
	By the construction of $h$, it follows immediately that 
	\begin{align*}
	\overline{H}_f \leq h
	\end{align*}
	in $B$ and $\overline{H}_f=h$ on $X$. For any $u \in \mathcal{U}_f$ and its Poisson modification $U=P(u, B)$, we have $u \geq U \geq \overline{H}_f$. Since $U \geq h$ on $X$ (which is dense in $B$) and $U, h$ are continuous in $B$, it follows that $U \geq h$ in $B$. Thus, $u \geq h$ in $B$ which implies that 
	\begin{align*}
	\overline{H}_f \geq h
	\end{align*}
	in $B$. Hence, $\overline{H}_f=h$ is $F$-harmonic in $\Omega$ and a similar argument for $\underline{H}_f$ completes the proof.
\end{proof}

We emphasize that although we proved that $F(D^2 \overline{H}_f)=0$ in $\Omega$, we cannot guarantee that $\overline{H}_f$ enjoys the boundary condition of the Dirichlet problem, $\overline{H}_f=f$ on $\partial \Omega$. To investigate the boundary behavior of the Perron solutions and ensure the solvability of the Dirichlet problem, we need to introduce further concepts, namely, a \textit{regular point} and a \textit{barrier function}.

\begin{definition}[A regular point]
	A boundary point $x_0 \in \partial \Omega$ is ($F$-)\textit{regular} with respect to $\Omega$, if 
	\begin{align*}
	\lim_{\Omega \ni y \to x_0} \overline{H}_f(y)=f(x_0) \quad \textrm{and} \quad \lim_{\Omega \ni y \to x_0} \underline{H}_f(y)=f(x_0)
	\end{align*}
	whenever $f \in C(\partial \Omega)$. An open and bounded set $\Omega$ is called \textit{regular} if each $x_0 \in \partial \Omega$ is a regular boundary point. 
\end{definition}

\begin{remark}
	Suppose that an operator $\mathcal{M}$ satisfies $\mathcal{M}[-u]=-\mathcal{M}[u]$; for example, any linear operator $L$ and $p$-Laplcian operators $\Delta_p$ possess this property. Then we have 
	\begin{align*}
	\overline{H}_f=-\underline{H}_{-f},
	\end{align*}
	and so in this case, we can equivalently call $x_0 \in \partial \Omega$ is regular if
	\begin{align*}
	\lim_{\Omega \ni y \to x_0} \overline{H}_f(y)=f(x_0)
	\end{align*}
	whenever $f \in C(\partial \Omega)$. Nevertheless, for the general fully nonlinear operator $F$, we do not have this property. Therefore, it seems that we have to require both conditions simultaneously, when we define a regular point for $F$. To the best of our knowledge, it is unknown whether the two conditions in the definition are redundant. One possible approach to show that only one condition is essential is to prove that $f$ is resolutive whenever $f$ is continuous on $\partial \Omega$; see Definition \ref{resolutive} for the definition of resolutivity.
\end{remark}

Before we define a barrier function, which characterizes a regular boundary point, we shortly deal with the resolutivity of boundary data:
\begin{definition}[Resolutivity] \label{resolutive}
	We say that a bounded function $f$ on $\partial \Omega$ is ($F$-)\textit{resolutive} if the upper and the lower Perron solutions $\overline{H}_f$ and $\underline{H}_f$ coincide in $\Omega$. When $f$ is resolutive, we write $H_f:=\overline{H}_f=\underline{H}_f$.
\end{definition}

\begin{lemma}
	Let $\Omega$ be a bounded open set of $\mathbb{R}^n$, let $f$ and $g$ be bounded functions on $\partial \Omega$, and let $c$ be any real number.
	\begin{enumerate}[(i)]
		\item If $f=c$ on $\partial \Omega$, then $f$ is resolutive and $H_f=c$ in $\Omega$.
		
		\item $\overline{H}_{f+c}=\overline{H}_f+c$ and $\underline{H}_{f+c}=\underline{H}_f+c$. If $f$ is resolutive, then $f+c$ is resolutive and $H_{f+c}=H_f+c$.
		
		\item If $c>0$, then $\overline{H}_{cf}=c\overline{H}_f$ and $\underline{H}_{cf}=c\underline{H}_f$. If $f$ is resolutive, then $cf$ is resolutive and $H_{cf}=cH_f$ for $c \geq 0$.
		
		\item If $f \leq g$, then $\overline{H}_f \leq \overline{H}_g$ and $\underline{H}_f \leq \underline{H}_g$.
	\end{enumerate}
\end{lemma}

Note that the resolutivity of $f$ does not imply
\begin{align*}
\lim_{y \to x}{H}_f(y)=f(x)
\end{align*}
for $x \in \partial \Omega$. However, the converse is true in some sense:

\begin{lemma} \label{resosuff}
	Let $\Omega$ be an open and bounded subset of $\mathbb{R}^n$ and $f$ be a bounded function on $\partial \Omega$. Suppose that there exists $F$-harmonic $h$ in $\Omega$ such that
	\begin{align*}
	\lim_{\Omega \ni y \to x}h(y)=f(x)
	\end{align*}
	for any $x \in \partial \Omega$. Then $\overline{H}_f=h=\underline{H}_f$. In particular, $f$ is resolutive.
\end{lemma}
 
\begin{proof}
	Since $h \in \mathcal{U}_f \cap \mathcal{L}_f$, we have $\overline{H}_f \leq h \leq \underline{H}_f.$
\end{proof}

\begin{lemma}\label{reslem2}
	If $u$ is a bounded $F$-superharmonic (or $F$-subharmonic) function on the bounded open set $\Omega$ such that $f(x)=\lim_{\Omega \ni y \to x}u(y)$ exists for all $x \in \partial \Omega$, then $f$ is a resolutive boundary function.
\end{lemma}

\begin{proof}
	Obviously, we have $u \in \mathcal{U}_f$ and so $\overline{H}_f \leq u$ in $\Omega$. Then since $\overline{H}_f$ is $F$-harmonic in $\Omega$ and 
	\begin{align*}
	 \limsup_{\Omega \ni y \to x}\overline{H}_f(y) \leq \lim_{\Omega \ni y \to x}u(y)=f(x),
	\end{align*}
	we have $\overline{H}_f \in \mathcal{L}_f$, which implies that $\overline{H}_f \leq \underline{H}_f$. Because $\underline{H}_f \leq \overline{H}_f$ always holds, we conclude that $f$ is resolutive. An analogous argument works for the $F$-subharmonic case.
\end{proof}

\subsection{Characterization of a regular point}
\begin{definition}[Barrier]
	Let $x_0 \in \partial \Omega$. A function $w^+ : \Omega \to \mathbb{R}$ [resp. $w^-$] is an \textit{upper barrier} [resp. \textit{lower barrier}] in $\Omega$ at the point $x_0$ if  
	\begin{enumerate}[(i)]
		\item $w^+$ [resp. $w^{-}$] is $F$-superharmonic [resp. $F$-subharmonic] in $\Omega$;
		\item $\liminf_{\Omega \ni y \to x}w^+(y)>0$ [resp. $\limsup_{\Omega \ni y \to x}w^-(y)<0$] for each $x \in \partial \Omega \setminus \{x_0\}$;
		\item $\lim_{\Omega \ni y \to x_0}w^+(y)=0$. [resp. $\lim_{\Omega \ni y \to x_0}w^-(y)=0$.]
	\end{enumerate}
\end{definition}

Observe that the maximum principle indicates that an upper barrier $w^+$ is positive in $\Omega$ and a lower barrier $w^-$ is negative in $\Omega$.  Moreover, under the condition (F2), $cw^+$ is still an upper barrier for any constant $c>0$ and an upper barrier $w^+$. See also \cite{RR94}.

Now we can deduce that a regular boundary point is characterized by the existence of upper and lower barriers.

\begin{theorem}\label{barchar}
	Let $x_0 \in \partial \Omega$. Then the following are equivalent:
	\begin{enumerate}[(i)]
		\item $x_0$ is regular;
		\item there exist an upper barrier and a lower barrier at $x_0$.
	\end{enumerate}
\end{theorem}

\begin{proof}
	\begin{enumerate}[]
		\item (ii) $\implies$ (i) For $f \in C(\partial \Omega)$ and $\varepsilon>0$, there is $\delta>0$ such that $|x-x_0| \leq \delta$ with $x \in \partial \Omega$ implies $|f(x)-f(x_0)| < \varepsilon$. Moreover, for $M:=\sup_{\partial \Omega} |f|$, there exists a large number $K>0$ such that 
		\begin{align*}
		K\cdot\liminf_{\Omega \ni y \to x} w^+(y) \geq 2M \ \textrm{for all $x \in \partial \Omega$ with $|x-x_0| \geq \delta$.}
		\end{align*}
		Here we used that $x \mapsto \liminf_{\Omega \ni y \to x}w^+(y)$ is lower semi-continuous on $\partial \Omega$.
		Then since $Kw^++f(x_0)+\varepsilon \in \mathcal{U}_f$, we have
		\begin{align*}
		\overline{H}_f(y) \leq Kw^+(y)+f(x_0)+\varepsilon,
		\end{align*}
		which implies that
		\begin{align*}
		\limsup_{\Omega \ni y \to x_0}\overline{H}_f(y) \leq f(x_0).
		\end{align*}
		An analogous argument leads to 
		\begin{align*}
		\liminf_{\Omega \ni y \to x_0}\underline{H}_f(y) \geq f(x_0).
		\end{align*}
		Since $\underline{H}_f \leq \overline{H}_f$, we conclude that
		\begin{align*}
		\lim_{\Omega \ni y \to x_0}\overline{H}_f(y)=f(x_0)=\lim_{\Omega \ni y \to x_0}\underline{H}_f(y),
		\end{align*}
		i.e. $x_0$ is a regular boundary point.

		\item (i) $\implies$ (ii) Define a distance function $d$ by 
		\begin{align*}
		d(y):=|y-x_0|^2
		\end{align*}
		so that $d$ is continuous, non-negative and $d(y)=0$ if and only if $y=x_0$. Moreover, since $D^2d=2I$, we have $F(D^2d)=2F(I)>0$, i.e. $d$ is $F$-subharmonic.
		
		 Then letting $w^+:=\underline{H}_d$, we have $w^+$ is $F$-harmonic  in $\Omega$ and it follows from $d \in \mathcal{L}_d$ that $w^+ \geq d$ in $\Omega$. Thus, for any $x \in \partial \Omega \setminus \{x_0\}$,
		 \begin{align*}
		 \liminf_{\Omega \ni y \to x}w^+(y)\geq d(x)=|x-x_0|^2>0.
		 \end{align*}
		 Furthermore, since $x_0$ is regular, we have
		\begin{align*}
		\lim_{\Omega \ni y \to x_0} w^+(y)=d(x_0)=0,
		\end{align*}
		and so $w^+$ is a desired upper barrier. The existence of a lower barrier is guaranteed by considering $\widetilde{d}(y):=-d(y)=-|y-x_0|^2$ and $w^-:=\overline{H}_{\widetilde{d}}$.
		
	\end{enumerate}
	
\end{proof}

Indeed, the barrier characterization is a \textit{local} property:
\begin{lemma} \label{local0}
	Let $x_0 \in \partial \Omega$ and $G \subset \Omega$ be open with $x_0 \in \partial G$. If $x_0$ is regular with respect to $\Omega$, then $x_0$ is regular with respect to $G$.
\end{lemma}

\begin{proof}
	By Theorem \ref{barchar}, there exist an upper barrier $w^+$ and a lower barrier $w^-$ with respect to $\Omega$ at $x_0$. Then $w^+|_G$ and $w^-|_G$ become the desired barriers with respect to $G$ at $x_0$. Again by Theorem \ref{barchar}, $x_0$ is regular with respect to $G$.
\end{proof}

\begin{lemma}\label{local}
	Let $x_0 \in \partial \Omega$ and $B$ be a ball containing $x_0$. Then $x_0$ is regular with respect to $\Omega$ if and only if $x_0$ is regular with respect to $B \cap \Omega$.
\end{lemma}

\begin{proof}
	By Lemma \ref{local0}, one direction is immediate. For the opposite direction, suppose that $x_0$ is regular with respect to $B \cap \Omega$. Then there exist an upper barrier $w^+$ and a lower barrier $w^-$ with respect to $B \cap \Omega$. If we let $m:=\min_{\partial B \cap \Omega} w^+>0$ (the minimum exists because $w^+$ is lower semi-continuous), then the pasting lemma, Lemma \ref{pasting}, shows that
	\begin{align*}
	s^+:=
	\left\{ \begin{array}{ll} 
	\min\{w^+, m\} & \textrm{in $B \cap \Omega$},\\
	m & \textrm{in $\Omega \setminus B$},
	\end{array} \right.
	\end{align*}
	is $F$-superharmonic in $\Omega$. One can easily verify that $s^+$ is an upper barrier with respect to $\Omega$ at $x_0$. Similarly, a lower barrier $s^-$ can be constructed.
\end{proof}

The barrier characterization leads to another useful corollary, which enables us to write $x_0$ is regular instead of $F$-regular, without ambiguity.
\begin{corollary}\label{ftildef}
	A boundary point $x_0 \in \partial \Omega$ is $F$-regular if and only if $x_0$ is $\widetilde{F}$-regular.
\end{corollary}

\begin{proof}
	Suppose that $x_0$ is $F$-regular. By Theorem \ref{barchar}, there exists an upper barrier $w^+_F$ and a lower barrier $w^-_F$. If we let $w^+_{\widetilde{F}}:=-w^-_F$ and $w^-_{\widetilde{F}}:=-w^+_F$, then $w^+_{\widetilde{F}}$ and $w^-_{\widetilde{F}}$ become an upper barrier and a lower barrier for $\widetilde{F}$, respectively. Therefore, again by Theorem \ref{barchar}, $x_0$ is $\widetilde{F}$-regular.
\end{proof}

Now we present one sufficient condition that guarantees a regular boundary point, namely the exterior cone condition. In Section 4, we suggest another sufficient condition, namely the Wiener criterion, which contains this exterior cone condition as a special case. 

\begin{theorem}[Exterior cone condition] \label{cone}
	Suppose that $\Omega$ satisfies an exterior cone condition at $x_0 \in \partial \Omega$. Then $x_0$ is a regular boundary point.
\end{theorem}

\begin{proof}
	The proof relies on the construction of a local barrier at $x_0$. See \cite{CKLS99, Mil67, Mil71} for details.
\end{proof}

\begin{corollary} \label{exhaust}
	All polyhedra and all balls are regular. Furthermore, every open set can be exhausted by regular open sets. Here a bounded open set $\Omega$ is called a \textit{polyhedron} if $\partial \Omega=\partial \overline{\Omega}$ and if $\partial \Omega$ is contained in a finite union of $(n-1)$-hyperplanes.
	
\end{corollary}

\begin{proof}
	Since polyhedra and balls satisfy the uniform exterior cone condition, the first assertion follows from \Cref{cone}. For the second assertion, exhaust $\Omega$ by domains $D_1 \subset\subset D_2 \subset\subset \cdots \subset\subset \Omega$. Then, since $\overline{D_j}$ is compact, there exists a finite union of open cubes $Q_{j_i} (\subset D_{j+1})$ that covers $\overline{D_j}$. Letting $P_j:=\bigcup_i \mathrm{int}\, \overline{Q_{j_i}}$ which is a polyhedron by the construction, we obtain the desired exhaustion.
\end{proof}

\section{Balayage and capacity}
\subsection{Balayage and capacity potential}
We define the \textit{lower semi-continuous regularization} $\hat{u}$ of any function $u : E \to [-\infty, \infty]$ by 
\begin{align*}
\hat{u}(x):=\lim_{r \to 0}\inf_{E \cap B_r(x)}u.
\end{align*}

\begin{lemma}\label{balayagesuper}
	Suppose that $\mathcal{F}$ is a family of $F$-superharmonic functions in $\Omega$, locally uniformly bounded below. Then the lower semi-continuous regularization $s$ of $\inf \mathcal{F}$, 
	\begin{align*}
	s(x)=\lim_{r \to 0} \inf_{B_r(x)} (\inf \mathcal{F}),
	\end{align*}
	is $F$-superharmonic in $\Omega$.
\end{lemma}

\begin{proof}
	Since $\mathcal{F}$ is locally uniformly bounded below, $s$ is lower semi-continuous. Fix an open $D \subset\subset \Omega$ and let $h \in C(\overline{D})$ be an $F$-harmonic function satisfying $h \leq s$ on $\partial D$. Then $h \leq u$ in $D$ whenever $u \in \mathcal{F}$. It follows from the continuity of $h$ that $h \leq s$ in $D$.
\end{proof}

\begin{definition}[Balayage and capacity potential]
	\hspace{0.1mm}
	\begin{enumerate}[(i)]
		\item 	For $\psi: \Omega \to (-\infty, \infty]$ which is locally bounded below, let \begin{align*}
		\Phi^{\psi}=\Phi^{\psi}(\Omega):=\{u : \text{$u$ is $F$-superharmonic in $\Omega$ and $u \geq \psi$ in $\Omega$}\}.
		\end{align*}
		Then the function 
		\begin{align*}
		R^{\psi}=R^{\psi}(\Omega):=\inf \Phi^{\psi}
		\end{align*}
		is called the \textit{reduced function} and its lower semi-continuous regularization
		\begin{align*}
		\hat{R}^{\psi}=\hat{R}^{\psi}(\Omega)
		\end{align*}
		is called the \textit{balayage} of $\psi$ in $\Omega$. By Lemma \ref{balayagesuper}, $\hat{R}^{\psi}$ is $F$-superharmonic in $\Omega$.
		
		\item If $u$ is a non-negative function on a set $E \subset \Omega$, we write 
		\begin{align*}
		\Phi_E^{u}=\Phi^{\psi}, \quad R_E^{u}=R^{\psi}, \quad \hat{R}_E^{u}=\hat{R}^{\psi},
		\end{align*}
		where 
		\begin{align*}
		\psi=	
		\left\{ \begin{array}{ll} 
		u & \textrm{in $E$},\\
		0 & \textrm{in $\Omega \setminus E$}.
		\end{array} \right.
		\end{align*}
		The function $\hat{R}_E^{u}$ is called the \textit{balayage of $u$ relative to $E$}. 
		
		\item In particular, we call the function $\hat{R}_E^1$ the  ($F$-)\textit{capacity potential} of $E$ in $\Omega$.
	\end{enumerate}
\end{definition}

\begin{remark}
	For an operator in divergence form, there exists an alternative method to define the capacity potential. For simplicity, suppose that the operator is given by the $p$-Laplacian. Let $\Omega$ be bounded and $K \subset \Omega$ be a compact set. For $\psi \in C_0^{\infty}(\Omega)$ with $\psi \equiv 1$ on $K$, the $p$-harmonic function $u$ in $\Omega \setminus K$ with $u-\psi \in W_0^{1, p}(\Omega \setminus K)$ is called the \textit{capacity potential} of $K$ in $\Omega$ and denoted by $\mathcal{R}(K, \Omega)$. Here note that $\mathcal{R}(K, \Omega)$ is independent of the particular choice of $\psi$ and the existence of the capacity potential is guaranteed by the variational method. Indeed, both definitions of capacity potentials coincide; see \cite[Chapter 9]{HKM93} for details.
\end{remark}

\begin{lemma} \label{balayage}
	The balayage $\hat{R}_E^u$ is $F$-harmonic in $\Omega \setminus \overline{E}$ and coincides with $R_E^u$ there. If, in addition, $u$ is $F$-superharmonic in $\Omega$, then $\hat{R}_E^u=u$ in the interior of $E$.
\end{lemma}

\begin{proof}
	Observe first that if $v_1$ and $v_2$ are in $\Phi_E^u$, then so is $\min\{v_1, v_2\}$. Hence, the family $\Phi_E^u$ is downward directed and we may invoke Choquet's topological lemma (see Lemma 8.3. in \cite{HKM93}): there is a decreasing sequence of functions $v_i \in \Phi_E^u$ with the limit $v$ such that
	\begin{align*}
	\hat{v}(x)=\hat{R}_E^u(x)
	\end{align*}
	for all $x \in \Omega$. 
	
	Next, we choose a ball $B \subset \subset \Omega \setminus \overline{E}$ and consider a Poisson modification $s_i=P(v_i, B)$. Then it follows that $s_i \in \Phi_E^u$ and $s_{i+1} \leq s_i \leq v_i$. Thus, we have
	\begin{align*}
	R_E^u \leq s:=\lim_{i \to \infty}s_i \leq v,
	\end{align*}
	which implies that $\hat{R}_E^u=\hat{v}=\hat{s}$. Moreover, since $s$ is $F$-harmonic in $B$ (Harnack convergence theorem, Theorem \ref{harnackct}), we know that $\hat{s}=s$. Therefore, we conclude that the balayage $\hat{R}_E^u$ is $F$-harmonic in $\Omega \setminus \overline{E}$. The second assertion of the lemma is rather immediate since $u \in \Phi_E^u$ if $u$ is $F$-superharmonic in $\Omega$.
\end{proof}

\begin{lemma}\label{brelot1}
	Let $K$ be a compact subset of $\Omega$ and consider $R_K^1=R_K^1(\Omega)$ and $\hat{R}_K^1=\hat{R}_K^1(\Omega)$.
	\begin{enumerate}[(i)]
		\item $0 \leq \hat{R}_K^1 \leq {R}_K^1 \leq 1$ in $\Omega$.
		\item $R_K^1=1$ in $K$.
		\item $R_K^1=\hat{R}_K^1$ in $(\partial K)^c$.
		\item $\hat{R}_K^1$ is $F$-superharmonic in $\Omega$ and $F$-harmonic in $\Omega \setminus K$.
	\end{enumerate}
\end{lemma}

\begin{proof}
	\begin{enumerate}[]
		\item (i) It immediately follows from the definition of $R_K^1$ and the comparison principle.
		
		\item (ii) Since $1 \in \Phi^{\psi}(\Omega)$, we have $R_K^1 \leq 1$ in $\Omega$. On the other hand, for any $u \in \Phi^{\psi}(\Omega)$, we have $u \geq \psi=1$ in $K$ and so $R_K^1 \geq 1$ in $K$.
		
		\item (iii), (iv) It immediately follows from Lemma \ref{balayage} and part (ii).
		
	\end{enumerate}
\end{proof}

The following theorem shows that the capacity potential can be understood as the upper Perron solution:
\begin{theorem} \label{cappoten}
	Suppose that $K$ is a compact subset of a bounded, open set $\Omega$ and that $u=\hat{R}_K^1(\Omega)$ is the capacity potential of $K$ in $\Omega$. Moreover, let $f$ be a function such that
	\begin{align*}
	f=	
	\left\{ \begin{array}{ll} 
	1 & \text{on $\partial K$},\\
	0 & \text{in $\partial \Omega$}.
	\end{array} \right.
	\end{align*}
	Then 
	\begin{align*}
	\hat{R}_{K}^{1}(\Omega)=\overline{H}_f(\Omega \setminus K)
	\end{align*}
	in $\Omega \setminus K$.
\end{theorem}

\begin{proof}
	Lemma \ref{balayage} shows that $\hat{R}_K^1=R_K^1$ in $\Omega \setminus K$. Then recall that
	\begin{align*}
	R_K^1(\Omega)=\inf\Phi_K^1=\inf\{v: \text{$v$ is $F$-superharmonic in $\Omega$ and $v \geq \psi$ in $\Omega$}  \} ,
	\end{align*} 
	where 
	\begin{align*}
	\psi=
	\left\{ \begin{array}{ll} 
	1 & \textrm{in $K$},\\
	0 & \textrm{in $\Omega \setminus K$},
	\end{array} \right.
	\end{align*}
	and
	\begin{align*}
	&\overline{H}_f(\Omega \setminus K)=\inf \mathcal{U}_f\\
	&=\inf\{v: \text{$v$ is $F$-superharmonic in $\Omega \setminus K$, $\liminf_{\Omega \setminus K \ni y \to x} v(y) \geq f(x)$ for each $x \in \partial (\Omega \setminus K)$}  \},
	\end{align*} 
	where 
	\begin{align*}
	f=	
	\left\{ \begin{array}{ll} 
	1 & \textrm{on $\partial K$},\\
	0 & \textrm{in $\partial \Omega$}.
	\end{array} \right.
	\end{align*}
	\begin{enumerate}[(i)]
		\item Suppose that $v \in \Phi_K^1$. Since $v \geq 0$ in $\Omega$, we have $\liminf_{\Omega \setminus K \ni y \to x}v(y) \geq 0=f(x)$ for $x \in \partial \Omega$. Moreover, since $v$ is lower semi-continuous, we have
		\begin{align*}
		\liminf_{\Omega \setminus K \ni y \to x}v(y) \geq \liminf_{\Omega \ni y \to x}v(y) \geq v(x) \geq 1=f(x),
		\end{align*}
		for $x \in \partial K$. Therefore, we conclude $v \in \mathcal{U}_f$, which implies that $\overline{H}_f(\Omega \setminus K) \leq R_K^1(\Omega)$ in $\Omega \setminus K$.
		
		\item Suppose that $v \in \mathcal{U}_f$. We consider $\overline{v}:=\min\{1, v\} \in \mathcal{U}_f$ so that $0 \leq \overline{v} \leq 1$ in $\Omega \setminus K$. Then, since $u \equiv 1$ is $F$-superharmonic in $\Omega$, the function
		\begin{align*}
	s=
	\left\{ \begin{array}{ll} 
	\min\{1,\overline{v}\}=\overline{v} & \text{in $\Omega \setminus K$},\\
	1 & \text{in $K$}
	\end{array} \right.
	\end{align*}
	is $F$-superharmonic in $\Omega$ by pasting lemma, \Cref{pasting}. Obviously, $s \in \Phi_K^1$ and so $R_K^1(\Omega) \leq \overline{v} \leq v$ in $\Omega \setminus K$.
	\end{enumerate}
	
\end{proof}

\subsection{Capacity}
In general, for an operator in divergence form, we consider a \textit{variational capacity}, which comes from minimizing the energy among admissible functions. On the other hand, for an operator in non-divergence form, we cannot consider the corresponding energy, and so we require an alternative approach to attain a proper notion of capacity. Our definition of a capacity is in the same context with Bauman \cite{Bau85} (for linear operators in non-divergence form) and Labutin \cite{Lab02b} (for the Pucci extremal operators). 

\begin{definition}[Non-variational capacity]
	For a ball $B=B_{2r}(x_0)$, we fix a ball $B'=B_{7/5r}(x_0) \subset B$ and a point $y_0=x_0+\frac{3}{2}r e_1$. Then we define a \textit{capacity} for fully nonlinear operator $F$ by
	\begin{align*}
	\mathrm{cap} (K,B)=\mathrm{cap}_F(K,B):=\inf\{u(y_0) : \text{$u$ is $F$-superharmonic in $B$, $u \geq 0$ in $B$, and $u \geq 1$ in $K$}\}
	\end{align*}
	whenever $K$ is a compact subset of $B'$.
\end{definition}

Comparing the definitions of capacity and capacity potential, we immediately notice that
\begin{align*}
\mathrm{cap} (K,B)=\hat{R}_K^1(B)(y_0).
\end{align*} 
Moreover, appealing to Theorem \ref{cappoten}, we further have
\begin{align*}
\mathrm{cap} (K,B)=\overline{H}_f(B \setminus K)(y_0),
\end{align*}
where the boundary data $f$ on $\partial (B \setminus K)$ is given by
\begin{align*}
f=	
\left\{ \begin{array}{ll} 
1 & \text{on $\partial K$},\\
0 & \text{in $\partial B$}.
\end{array} \right.
\end{align*}
Finally, considering Harnack inequality for $\hat{R}_K^1(B)$ on the sphere $\partial B_{3r/2}(x_0)$, we notice that capacities defined for different choices of $y_0 \in \partial B_{3r/2}(x_0)$ are comparable.

\begin{lemma}[Properties of capacity]\label{brelot2}
	Fix a ball $B=B_{2r}(x_0)$. Then the set function $K \mapsto \mathrm{cap} (K,B)$, where $K$ is a compact subset of $B'=B_{7/5r}(x_0)$, enjoys the following properties:
	\begin{enumerate}[(i)]
		\item $0 \leq \mathrm{cap} (K,B) \leq 1$.
		\item If $K_1 \subset K_2 \subset B'$, then
		\begin{align*}
		\mathrm{cap} (K_1 ,B) \leq \mathrm{cap} (K_2, B).
		\end{align*}
		\item If a monotone sequence of compact sets $\{K_j\}_{j=1}^{\infty}$ satisfies $B' \supset K_1 \supset K_2 \supset \cdots$, then
		\begin{align*}
		\mathrm{cap} (K, B)=\lim_{j \to \infty}\mathrm{cap} (K_j, B), \quad \text{for $K:=\bigcap_{j=1}^{\infty}K_j$.}
		\end{align*}
		\item (Subadditivity) We further suppose that $F$ is convex. If $K_1$ and $K_2$ are compact subsets of $B'$, then
		\begin{align*}
		\mathrm{cap}(K_1 \cup K_2, B)  \leq \mathrm{cap} (K_1, B)+\mathrm{cap} (K_2,B).
		\end{align*}
	\end{enumerate}
\end{lemma}

\begin{proof}
	\begin{enumerate}[(i)]
		\item Recalling Lemma \ref{brelot1}, we have $0 \leq \mathrm{cap} (K,B) \leq 1$. 
		
		\item If $K_1 \subset K_2$, then $\Phi_{K_2}^1 \subset \Phi_{K_1}^1$ and so $\mathrm{cap} (K_1,B) \leq \mathrm{cap} (K_2,B)$.
		
		\item Since $\mathrm{cap}(K_j, B) \geq \mathrm{cap}(K, B)$ by (ii), it is immediate that 
		\begin{align*}
		\mathrm{cap}(K, B) \leq \lim_{j \to \infty} \mathrm{cap}(K_j, B).
		\end{align*}
		 For the reversed inequality, fix small $\varepsilon>0$ and $u \in \Phi_K^1(B)$. If $j$ is large enough, then $K_j \subset \{u \geq 1-\varepsilon\}$ and so 
		\begin{align*}
		\lim_{j \to \infty}\mathrm{cap}(K_j, B) \leq \mathrm{cap}(\{u \geq 1-\varepsilon\}, B) \leq \frac{1}{1-\varepsilon}u(y_0).
		\end{align*}
		Letting $\varepsilon \to 0^+$ and taking infimum for $u \in \Phi_K^1(B)$, we conclude that 
		\begin{align*}
		\lim_{j \to \infty} \mathrm{cap}(K_j, B) \leq \mathrm{cap}(K, B).
		\end{align*}
		
		\item Let $v_1 \in \Phi_{K_1}^1(B)$ and $v_2 \in \Phi_{K_2}^1(B)$. Since $F$ is convex, we can apply \cite[Theorem 5.8]{CC95} to obtain $\frac{1}{2}(v_1+v_2)$ is $F$-superharmonic in $B$. Moreover, it follows from the assumption (F2) that $v_1+v_2 \in \Phi_{K_1 \cup K_2}^1(B)$ and so $R_{K_1 \cup K_2}^1(B) \leq v_1+v_2$. Putting the infimum on this inequality and evaluating at $y_0$, we conclude that 
		\begin{align*}
		\mathrm{cap}(K_1 \cup K_2, B)  \leq \mathrm{cap} (K_1, B)+\mathrm{cap} (K_2,B).
		\end{align*}
	\end{enumerate}
\end{proof}

We would like to remove the restriction of compact sets when defining a capacity. For this purpose, when $U \subset B'$ is open, we set the \textit{inner capacity} 
\begin{align*}
\mathrm{cap}_{\ast}(U,B):=\sup_{K \subset U,\, K \, \text{compact}} \mathrm{cap} (K,B).	
\end{align*}
Then for an arbitrary set $E \subset B'$, we set the \textit{outer capacity}
\begin{align*}
\mathrm{cap}^{\ast}(E,B):=\inf_{E \subset U \subset B',\, U \, \text{open}} \mathrm{cap}_{\ast}(U,B).
\end{align*}

\begin{lemma}\label{capacitable}
	Fix a ball $B=B_{2r}(x_0)$. For a compact subset $K$ of $B'=B_{7r/5}(x_0)$, we have 
	\begin{align*}
	\mathrm{cap}(K, B)=\mathrm{cap}^{\ast}(K, B).
	\end{align*}
	In other words, there is no ambiguity in having two different definitions for the capacity of compact sets.
\end{lemma}

\begin{proof}
	\begin{enumerate}[(i)]
		\item For any open set $U$ satisfying $K \subset U \subset B'$, the definition of the inner capacity yields that
		\begin{align*}
		\mathrm{cap}(K, B) \leq \mathrm{cap}_{\ast}(U, B).
		\end{align*}
		By taking the infimum over such $U$, we conclude that
		\begin{align*}
			\mathrm{cap}(K, B) \leq \mathrm{cap}^{\ast}(K, B).
		\end{align*}
		
		\item Define a sequence of compact sets $\{K_j\}_{j=1}^{\infty}$ by 
		\begin{align*}
		K_j:=\{x \in \mathbb{R}^n: \mathrm{dist}(x, K) \leq 1/j\},
		\end{align*}
		and a sequence of open sets $\{U_j\}_{j=1}^{\infty}$ by 
		\begin{align*}
		U_j:=\{x \in \mathbb{R}^n: \mathrm{dist}(x, K) < 1/j\}.
		\end{align*}
		We may assume $K_1 \subset B'$.	Then we have
		\begin{align*}
		B' \supset K_1 \supset U_1 \supset K_2 \supset U_2 \supset \cdots \supset K, \quad \text{and} \quad K=\bigcap_j K_j.
		\end{align*}
		Applying Lemma \ref{brelot2} (ii), it follows that
		\begin{align*}
			\mathrm{cap}_{\ast}(U_j, B) \leq \mathrm{cap}(K_j, B).
		\end{align*}
		By the definition of outer capacity,
		\begin{align*}
		\mathrm{cap}^{\ast}(K, B) \leq \mathrm{cap}_{\ast}(U_j, B) \leq \mathrm{cap}(K_j, B), \quad \text{for any $j \in \mathbb{N}$.}
		\end{align*}
		Now letting $j \to \infty$, Lemma \ref{brelot2} (iii) leads to 
		\begin{align*}
		\mathrm{cap}^{\ast}(K, B) \leq \mathrm{cap}(K, B).
		\end{align*}
	\end{enumerate}
\end{proof}

Roughly speaking, we have the following correspondance:
	\begin{align*}
	\text{the variational capacity} &\longleftrightarrow \text{divergence operator},\\
	\text{the height capacity} &\longleftrightarrow \text{non-divergence operator}.
	\end{align*}
	In the following lemma, we explain why the definition of height capacity is reasonable in some sense. In other words, we claim that for the Laplacian operator $\Delta$, two definitions of capacity are comparable. 

\begin{lemma}[The variational capacity and the height capacity]
	Suppose $n \geq 3$ and fix two balls $B=B_{2r}(x_0)$, $B'=B_{7r/5}(x_0)$ and a point $y_0=\frac{3}{2}re_1+x_0 \in \partial B_{3r/2}(x_0)$. Then for any compact set $K \subset B'$, we have
	\begin{align*}
	\mathrm{cap}_{\Delta, \mathrm{var}}(K, B) \sim \mathrm{cap}_{\Delta, \mathrm{height}}(K, B)\, r^{n-2},
	\end{align*}
	where the comparable constant depends only on $n$. 
\end{lemma}

\begin{proof}
	We may assume $x_0=0$. We denote by $u$ the capacity potential with respect to $K$ in $B$. Note that $u$ is harmonic in $B \setminus K$.
	
	We begin with the variational capacity:
	\begin{align*}
	\mathrm{cap}_{\Delta, \mathrm{var}}(K, B)=\int_{B \setminus K}|\nabla u|^2 \,\mathrm{d}x=\int_{\partial K} \frac{\partial u}{\partial \mathbf{n}} \,\mathrm{d}s=-\int_{\partial B} \frac{\partial u}{\partial \mathbf{n}} \,\mathrm{d}s.
	\end{align*}
	Here we applied the divergence theorem and used the behavior of $u$ on the boundary.
	
	On the other hand, recalling the definition of height capacity, we have
	\begin{align*}
	\mathrm{cap}_{\Delta, \mathrm{height}}(K, B)=u(y_0).
	\end{align*}
	By Harnack inequality, there exist constants $c_1, c_2>0$ which only depend on $n$ such that
	\begin{align*}
	c_1 u(y_0) \leq u(x) \leq c_2u(y_0) \quad \text{for any $x \in \partial B_{3r/2}$}.
	\end{align*}
	Thus, if we set $m_-:=\min_{\partial B_{3r/2}}u$ and $m_+:=\max_{\partial B_{3r/2}}u$, then we have 
	\begin{align*}
	c_1 \mathrm{cap}_{\Delta, \mathrm{height}}(K, B) \leq m_- \leq m_+ \leq c_2\mathrm{cap}_{\Delta, \mathrm{height}}(K, B).
	\end{align*}
	Moreover, we consider two barriers $h^{\pm}$ which solve the Dirichlet problem in $B_{2r} \setminus B_{3r/2}$:
	\begin{align*}
	\left\{ \begin{array}{ll} 
	\Delta h^{\pm}=0 & \text{in $B_{2r} \setminus B_{3r/2}$},\\
	h^{\pm}=m_{\pm} & \text{on $\partial B_{3r/2}$},\\
	h^{\pm}=0 & \textrm{on $\partial B_{2r}$}.
	\end{array} \right.
	\end{align*}
	Indeed, using the homogeneous solution $V(x)=|x|^{2-n}$, one can compute $h^{\pm}$ explicitly:
	\begin{align*}
	h^{\pm}(x)=m_{\pm} \cdot \frac{|x|^{2-n}-(2r)^{2-n}}{(3r/2)^{2-n}-(2r)^{2-n}}.
	\end{align*}
	Then the comparison principle between $u$ and $h^{\pm}$ leads to
	\begin{align*}
	h^{-} \leq u \leq h^+ \quad \text{in $B_{2r} \setminus B_{3r/2}$},
	\end{align*}
	and so
	\begin{align*}
	\frac{c(n)\,m_-}{r}=-\frac{\partial h^-}{\partial \mathbf{n}} \leq -\frac{\partial u}{\partial \mathbf{n}} \leq -\frac{\partial h^+}{\partial \mathbf{n}}=\frac{c(n)\,m_+}{r} \quad \text{on $\partial B$}.
	\end{align*}
	Therefore, we conclude that
	\begin{align*}
	c_1(n)r^{n-2}\mathrm{cap}_{\Delta, \mathrm{height}}(K, B) \leq  \mathrm{cap}_{\Delta, \mathrm{var}}(K, B) \leq c_2(n)r^{n-2}\mathrm{cap}_{\Delta, \mathrm{height}}(K, B).
	\end{align*}
\end{proof}

Next, we estimate the capacity of a ball $B_{\rho}$ with respect to the larger ball $B_{2r}$. Indeed, the capacity of a ball can capture the growth rate of the homogeneous solution $V$ of $F$. 

\begin{lemma}[Capacitary estimate for balls]\label{capball}
	Let $B=B_{2r}(x_0)$, $B'=B_{\frac{7}{5}r}(x_0)$ and $y_0=x_0+\frac{3}{2}re_1$. Then for any $0< \rho < \frac{7}{5}r$, there exists a constant $c=c(n, \lambda, \Lambda)>0$ which is independent of $r$ and $\rho$ such that
	\begin{enumerate}[(i)]
		\item ($\alpha^{\ast}>0$)
		\begin{align*}
		\frac{1}{c}\, \frac{r^{-\alpha^{\ast}}}{\rho^{-\alpha^{\ast}}-(2r)^{-\alpha^{\ast}}} \leq \mathrm{cap}_F(\overline{B_{\rho}(x_0)}, B_{2r}(x_0)) \leq  \frac{cr^{-\alpha^{\ast}}}{\rho^{-\alpha^{\ast}}-(2r)^{-\alpha^{\ast}}}.
		\end{align*}
		
		\item ($\alpha^{\ast}<0$)
		\begin{align*}
		\frac{1}{c}\, \frac{r^{-\alpha^{\ast}}}{(2r)^{-\alpha^{\ast}}-\rho^{-\alpha^{\ast}}} \leq \mathrm{cap}_F(\overline{B_{\rho}(x_0)}, B_{2r}(x_0)) \leq  \frac{cr^{-\alpha^{\ast}}}{(2r)^{-\alpha^{\ast}}-\rho^{-\alpha^{\ast}}}.
		\end{align*}
		
		\item ($\alpha^{\ast}=0$)
		\begin{align*}
		\frac{1}{c}\,\frac{1}{\log(2r)-\log{\rho}} \leq \mathrm{cap}_F(\overline{B_{\rho}(x_0)}, B_{2r}(x_0))    \leq \frac{c}{\log(2r)-\log{\rho}}.
		\end{align*} 
	\end{enumerate}

\end{lemma}

\begin{proof}
	We may assume $x_0=0$. Applying the argument after the definition of a capacity, we have 
	\begin{align*}
	\mathrm{cap}_F(\overline{B_{\rho}}, B_{2r})=\hat{R}_{\overline{B_{\rho}}}(B_{2r})(y_0)=\overline{H}_f(B_{2r}\setminus \overline{B_{\rho}})(y_0), 
	\end{align*}
	where the boundary data $f$ is given by
	\begin{align*}
	f=	
	\left\{ \begin{array}{ll} 
	1 & \text{on $\partial B_{\rho}$},\\
	0 & \text{in $\partial B_{2r}$}.
	\end{array} \right.
	\end{align*}
	Moreover, since a ball is a regular domain, we can write $\overline{H}_f(B_{2r}\setminus \overline{B_{\rho}})=v$ where $v$ is the unique solution of the Dirichlet problem 
	\begin{align*}
	\left\{ \begin{array}{ll} 
	F(D^2v)=0 & \text{in $B_{2r}\setminus \overline{B_{\rho}}$},\\
	v=1 & \text{on $\partial B_{\rho}$},\\
	v=0 & \text{in $\partial B_{2r}$}.
	\end{array} \right.
	\end{align*}
	Note that $\overline{H}_f(B_{2r}\setminus \overline{B_{\rho}})$ is continuous upto the boundary. We now split three cases according to the sign of $\alpha^{\ast}(F)$.
	\begin{enumerate}[(i)]
		\item ($\alpha^{\ast}>0$) In this case, for the homogeneous solution $V(x)=|x|^{-\alpha^{\ast}}V\Big(\frac{x}{|x|}\Big)$, denote 
		\begin{align*}
		V_+:=\max_{|x|=1}V(x) \quad\text{and}\quad 	V_-:=\min_{|x|=1}V(x)
		\end{align*}
		and choose two points $x_+, x_-$ with $|x_+|=1=|x_-|$ so that 
		\begin{align*}
		V(x_+)=V_+ \quad\text{and}\quad V(x_-)=V_-.
		\end{align*}
		We define two functions
		\begin{align*}
		v^+(x):=\frac{V(x)-(2r)^{-\alpha^{\ast}}V_-}{[\rho^{-\alpha^{\ast}}-(2r)^{-\alpha^{\ast}}]V_-} \quad\text{and}\quad v^-(x):=\frac{V(x)-(2r)^{-\alpha^{\ast}}V_+}{[\rho^{-\alpha^{\ast}}-(2r)^{-\alpha^{\ast}}]V_+}. 
		\end{align*}
		Then we have
		\begin{align*}
		\text{$F(D^2v^+)=0=F(D^2v^-)$ in $B_{2r} \setminus \overline{B_{\rho}}$,  }\\
		\text{$v^+ \geq 1$ on $\partial B_{\rho}$ and $v^+ \geq 0$ on $\partial B_{2r}$, } \\
		\text{$v^- \leq 1$ on $\partial B_{\rho}$ and $v^- \leq 0$ on $\partial B_{2r}$. }
		\end{align*}
		Thus, the comparison principle yields that
		\begin{align*}
		v^- \leq v=\overline{H}_f(B_{2r}\setminus \overline{B_{\rho}})=\hat{R}_{\overline{B_{\rho}}}(B_{2r}) \leq v^+ \quad \text{in $B_{2r} \setminus \overline{B_{\rho}}$}.
		\end{align*}
		Finally, applying Harnack inequality for $v$ on $\partial B_{3r/2}$, there exists a constant $c_1>0$ which is independent of $r>0$ such that
		\begin{align*}
		\frac{1}{c_1} v\Big(\frac{3rx_+}{2}\Big) \leq v(y_0) \leq c_1v\Big(\frac{3rx_-}{2}\Big).
		\end{align*}
		Therefore, we have the desired upper bound:
		\begin{align*}
		\mathrm{cap}_F(\overline{B_{\rho}}, B_{2r})=v(y_0) \leq c_1 v\Big(\frac{3rx_-}{2}\Big) \leq c_1 v^+\Big(\frac{3rx_-}{2}\Big)=&c_1 \frac{(3r/2)^{-\alpha^{\ast}}-(2r)^{-\alpha^{\ast}}}{\rho^{-\alpha^{\ast}}-(2r)^{-\alpha^{\ast}}} 
		\\
		=& \frac{cr^{-\alpha^{\ast}}}{\rho^{-\alpha^{\ast}}-(2r)^{-\alpha^{\ast}}}.
		\end{align*}
		Similarly, we derive the lower bound:
		\begin{align*}
		\mathrm{cap}_F(\overline{B_{\rho}}, B_{2r})=v(y_0) \geq \frac{1}{c_1} v\Big(\frac{3rx_+}{2}\Big) \geq \frac{1}{c_1} v^-\Big(\frac{3rx_+}{2}\Big)=&\frac{1}{c_1} \frac{(3r/2)^{-\alpha^{\ast}}-(2r)^{-\alpha^{\ast}}}{\rho^{-\alpha^{\ast}}-(2r)^{-\alpha^{\ast}}}		\\
		=&\frac{1}{c} \frac{r^{-\alpha^{\ast}}}{\rho^{-\alpha^{\ast}}-(2r)^{-\alpha^{\ast}}}.
		\end{align*}

		\item ($\alpha^{\ast}<0$) For simplicity, we assume that the upward-pointing homogeneous solution is given by 
		\begin{align*}
		V(x)=-|x|^{-\alpha^{\ast}}.
		\end{align*}
		Then we can explicitly write the capacity potential:
		\begin{align*}
		v(x)=\frac{(2r)^{-\alpha^{\ast}}-|x|^{-\alpha^{\ast}}}{(2r)^{-\alpha^{\ast}}-\rho^{-\alpha^{\ast}}}.
		\end{align*}
		Thus,  
		\begin{align*}
		\mathrm{cap}_F(\overline{B_{\rho}}, B_{2r})=v(y_0) \sim  \frac{r^{-\alpha^{\ast}}}{(2r)^{-\alpha^{\ast}}-\rho^{-\alpha^{\ast}}}.
		\end{align*}
		For general $V$, we can compute by a similar argument as in part (i). For example, if $V(x)=-|x|^{-\alpha^{\ast}}V\Big(\frac{x}{|x|}\Big)$, then define
		\begin{align*}
		v^+(x):=\frac{(2r)^{-\alpha^{\ast}}V_++V(x)}{[(2r)^{-\alpha^{\ast}}-\rho^{-\alpha^{\ast}}]V_+} \quad\text{and}\quad v^-(x):=\frac{(2r)^{-\alpha^{\ast}}V_-+V(x)}{[(2r)^{-\alpha^{\ast}}-\rho^{-\alpha^{\ast}}]V_-}. 
		\end{align*}
		
		\item ($\alpha^{\ast}=0$) Again for simplicity, we may assume the upward-pointing homogeneous solution is given by
		\begin{align*}
		V(x)=-\log|x|.
		\end{align*}
		Similarly, we can explicitly write the capacity potential:
		\begin{align*}
		v(x)=\frac{\log(2r)-\log{|x|}}{\log(2r)-\log{\rho}}.
		\end{align*}
		Thus,
		\begin{align*}
		\mathrm{cap}_F(\overline{B_{\rho}}, B_{2r})=v(y_0) \sim  \frac{1}{\log(2r)-\log{\rho}}.
		\end{align*}
		For general $V$, we can compute by a similar argument as in part (i). For example, if $V(x)=V\Big(\frac{x}{|x|}\Big)-\log|x|$, then define
		\begin{align*}
		v^+(x):=\frac{\log(2r)-V_-+V(x)}{\log(2r)-\log{\rho}} \quad\text{and}\quad v^-(x):=\frac{\log(2r)-V_++V(x)}{\log(2r)-\log{\rho}}. 
		\end{align*}
	\end{enumerate}

\end{proof}

We can observe that the capacity of a single point is determined according to the sign of the scaling exponent $\alpha^{\ast}(F)$. In fact, one can expect the results of the following lemma taking $\rho \to 0^+$ in the capacitary estimate, Lemma \ref{capball}. 
\begin{lemma}\label{sinpoint}
	For $z_0 \in \mathbb{R}^n$, choose a ball $B=B_{2r}(x_0)$ so that $z_0 \in B'=B_{7r/5}(x_0)$. 
	\begin{enumerate}[(i)]
		\item If $\alpha^{\ast}(F)\geq0$, then
		\begin{align*}
		\mathrm{cap}_F (\{z_0\}, B)=0.
		\end{align*}
		
		\item If $\alpha^{\ast}(F)<0$, then
		\begin{align*}
		\mathrm{cap}_F (\{z_0\}, B)>0.
		\end{align*}
	\end{enumerate}
\end{lemma}

\begin{proof}
	\begin{enumerate}[(i)]
		\item Let
		\begin{align*}
		V(x)=
		\left\{ \begin{array}{ll} 
		|x|^{-\alpha^{\ast}} V\Big(\frac{x}{|x|}\Big)	 & \text{if $\alpha^{\ast}>0$},\\
		-\log|x|+V\Big(\frac{x}{|x|}\Big) & \text{if $\alpha^{\ast}=0$}.
		\end{array} \right.
		\end{align*}
		be the homogeneous solution of $F$. Then for $m:=\min_{x \in \partial B} V(x-z_0)$ and any $\varepsilon>0$, we have
		\begin{align*}
		\varepsilon \cdot [V(x-z_0)-m] \in \Phi_{\{z_0\}}^1 
		\end{align*} 
		due to the minimum principle and $\lim_{x \to z_0}V(x-z_0)=\infty$. Thus,
		\begin{align*}
		\mathrm{cap}(\{z_0\}, B)=\hat{R}_{\{z_0\}}^1(y_0)=R_{\{z_0\}}^1(y_0) \leq \varepsilon \cdot [V(y_0-z_0)-m].
		\end{align*} 
		Since $\varepsilon>0$ is arbitrary, we finish the first part of proof.
		
		\item Let $V(x)=-|x|^{-\alpha^{\ast}}V\Big(\frac{x}{|x|}\Big)$ be the homogeneous solution of $F$. Then for $\max_{x \in \partial B}V(x-z_0)=:-M<0$, we consider
		\begin{align*}
		u(x):=1+\frac{V(x-z_0)}{M}.
		\end{align*}
		Since $\sup_{\partial B}u=0$ and $V$ is a homogeneous function, we have $\sup_{\partial B_{7/5r}}u>0$. On the other hand, recalling Theorem \ref{cappoten}, 
		\begin{align*}
		\hat{R}^1_{\{z_0\}}=\overline{H}_f(\Omega \setminus \{z_0\}) \geq \underline{H}_f(\Omega \setminus \{z_0\}),
		\end{align*}
		where the boundary data $f$ is given by
		\begin{align*}
		f(x)=	
		\left\{ \begin{array}{ll} 
		1 & \text{if $x=z_0$},\\
		0 & \text{if $x \in \partial B$}.
		\end{array} \right.
		\end{align*}
		Then $u \in \mathcal{L}_f$ and so $\underline{H}_f(\Omega \setminus \{z_0\}) \geq u$. Therefore, we conclude that 
		\begin{align*}
			\sup_{\partial B_{7/5r}} \hat{R}^1_{\{z_0\}} >0
		\end{align*}
		and by Harnack inequality, $\mathrm{cap}(\{z_0\}, B)>0$ as desired.
	\end{enumerate}
	
\end{proof}

\subsection{Capacity zero sets}
\begin{definition}
	A set $E$ in $\mathbb{R}^n$ is said to be of ($F$-)\textit{capacity zero}, or to have ($F$-)\textit{capacity zero} if 
	\begin{align*}
	\mathrm{cap}_F(E, B)=0
	\end{align*}
	whenever $E \subset B' \subset B$. In this case, we write $\mathrm{cap}_F E=0$.
\end{definition}

According to Lemma \ref{sinpoint} (i), we immediately notice that every single point is of $F$-capacity zero if $\alpha^{\ast}(F) \geq 0$. Indeed, we are going to show that: to check whether a compact set $K$ is of capacity zero or not, it is enough to test with respect to one ball $B$ (Corollary \ref{capzero}). For this purpose, we require the following version of a capacitary estimate, called ``comparable lemma".

\begin{lemma}[Comparable lemma] \label{capest}
	If $K \subset B'=B_{7r/5}$ and $0<r\leq s\leq 2r$, then there exists a universal constant $c>0$ such that
		\begin{align*}
		\frac{1}{c} \, \mathrm{cap}_F (K, B_{2r}) \leq \mathrm{cap}_F(K, B_{2s}) \leq c \, \mathrm{cap}_F (K, B_{2r}).
		\end{align*}
\end{lemma}

\begin{proof}
	We may assume $x_0=0$. We claim that for $0<r \leq s \leq \frac{21}{20}r$, we have
	\begin{align*}
	\frac{1}{c}\,\mathrm{cap}_F(K, B_{2r}) \leq \mathrm{cap}_F(K, B_{2s}) \leq c\, \mathrm{cap}_F (K, B_{2r}).
	\end{align*}
	Indeed, we may iterate this inequality finitely many times to conclude the desired inequality for $0<r\leq s\leq 2r$. Moreover, let $y_r=\frac{3}{2}re_1$, $y_s=\frac{3}{2}se_1$ and denote  $u_r:=\hat{R}_K^1(B_{2r})$, $u_s:=\hat{R}_K^1(B_{2s})$. By the definition of the capacity potential, it is immediate that $u_r \leq u_s$ in $B_{2r}$. In particular, we have
	\begin{align*}
	\mathrm{cap}_F(K, B_{2r})=u_r(y_r) \leq u_s(y_r).
	\end{align*}
	On the other hand, an application of Harnack inequality for $u_s$ (in a small neighborhood of $B_{3s/2} \setminus B_{10s/7}$) yields that there exists a constant $c>0$ which is independent of the choice of $r$ and $s$ such that
	\begin{align*}
	u_s(y_r) \leq cu_s(y_s)=c \, \mathrm{cap}_F(K, B_{2s}).
	\end{align*}
	Here note that $|y_r|=\frac{3}{2}r\geq \frac{10}{7}s>\frac{7}{5}s$ and $R_K^1(B_{2s})$ is $F$-harmonic in $B_{2s} \setminus B_{7s/5}$ and $B_{3s/2} \setminus B_{10s/7} \subset B_{2s} \setminus B_{7s/5}$. Therefore, it finishes the proof for the first inequality.
	
	Next, for the second inequality, we first assume that $\alpha^{\ast}(F) >0$ and the homogeneous solution is given by $V(x)=|x|^{-\alpha^{\ast}}$ (for computational simplicity) and let
	\begin{align*}
		M:=\max_{\partial B_{2r}}u_s \in [0,1).
	\end{align*}
	Then recalling Theorem \ref{cappoten}, the comparison principle yields that
	\begin{align}\label{com1}
	(1-M)u_r+M \geq u_s \quad \text{in $B_{2r} \setminus K$.}
	\end{align}
	Now choose $z \in \partial B_{3r/2}$ so that 
	\begin{align*}
	u_s(z)=\max_{\partial B_{3r/2}}u_s=:M_1.
	\end{align*}
	Then it can be easily checked that the function 
	\begin{align*}
	w(x):=M_1 \cdot\frac{|x|^{-\alpha^{\ast}}-(2s)^{-\alpha^{\ast}}}{(3r/2)^{-\alpha^{\ast}}-(2s)^{-\alpha^{\ast}}}
	\end{align*}
	is $F$-harmonic in $B_{2s} \setminus B_{3r/2}$ and by the comparison principle, $w \geq u_s$ in $B_{2s} \setminus B_{3r/2}$. (here again note that $\frac{7}{5}s <\frac{3}{2}r$.) In particular,
	\begin{align*}
	&M_1 \cdot\frac{(2r)^{-\alpha^{\ast}}-(2s)^{-\alpha^{\ast}}}{(3r/2)^{-\alpha^{\ast}}-(2s)^{-\alpha^{\ast}}} \geq M,\\
	&M_1 \cdot\frac{(3s/2)^{-\alpha^{\ast}}-(2s)^{-\alpha^{\ast}}}{(3r/2)^{-\alpha^{\ast}}-(2s)^{-\alpha^{\ast}}} \geq u_s\Big(\frac{3}{2}se_1\Big)=\mathrm{cap}_F(K, B_{2s}).	
	\end{align*}
	Since $(3r/2)^{-\alpha^{\ast}}-(2r)^{-\alpha^{\ast}} \geq (3s/2)^{-\alpha^{\ast}}-(2s)^{-\alpha^{\ast}}$ or equivalently,
	\begin{align*}
	(3r/2)^{-\alpha^{\ast}}-(2s)^{-\alpha^{\ast}} \geq [(3s/2)^{-\alpha^{\ast}}-(2s)^{-\alpha^{\ast}}]+[(2r)^{-\alpha^{\ast}}-(2s)^{-\alpha^{\ast}}],
	\end{align*}
	we obtain
	\begin{align} \label{com2}
	u_s(z)=M_1 \geq M+\mathrm{cap}_F(K, B_{2s}).
	\end{align}
	Moreover, by \eqref{com1} and \eqref{com2}, we have $u_r(z) \geq (1-M)u_r(z)\geq \mathrm{cap}_F(K, B_{2s})$ and then Harnack inequality leads to
	\begin{align*}
	\mathrm{cap}_F(K, B_{2s}) \leq c\,\mathrm{cap}_F(K, B_{2r}),
	\end{align*}
	for constant $c>0$ which is independent of $r$ and $s$. Finally, for the general homogeneous solution or the case of $\alpha^{\ast}(F)\leq0$, one can follow the idea of Lemma \ref{capball}.
\end{proof}

\begin{corollary}\label{capzero}
	Suppose that $\mathrm{cap}(K, B)=0$ for $K \subset B' \subset B$. Then 
	\begin{enumerate}[(i)]
		\item for any ball $B_1$ such that $K \subset B_1'$ and $B_1 \subset B'$, we have 
		\begin{align*}
		\mathrm{cap}(K, B_1)=0;
		\end{align*}
		
		\item for any ball $B_2$ such that $B_2' \supset B$, we have 
		\begin{align*}
		\mathrm{cap}(K, B_2)=0;
		\end{align*}
		
		\item $K$ is of $F$-capacity zero.
	\end{enumerate}
\end{corollary}

\begin{proof}
	\begin{enumerate}[(i)]
		\item Apply the first inequality of Lemma \ref{capest} finitely many times.
		
		\item Apply the second inequality of Lemma \ref{capest} finitely many times.
		
		\item It is an immediate consequence of (i) and (ii).
	\end{enumerate}
	
\end{proof}

%
%

Now we shortly illustrate the potential theoretic meaning of capacity zero sets, at least for convex operators $F$. In the end, $F$-capacity zero sets are `negligible' in view of the fully nonlinear operator $F$; i.e. $F$-capacity really measures the size of given sets in a suitable way to interpret the corresponding PDE.

\begin{definition}[Polar sets]
	A compact set $K$ is called \textit{$F$-polar}, or simply \textit{polar}, if there exist an open ball $B_{2r}$ with $K \subset B_{7r/5}$, and $F$-superharmonic function $u$ in $B_{2r}$ such that $u|_{K}=\infty$.
\end{definition}

\begin{lemma}\label{polar}
	Suppose that $K$ is a compact set in $B_{7r/5}$ and $F$ is convex. Then the followings are equivalent:
	\begin{enumerate}[(i)]
		\item $K$ is polar.
		\item $\mathrm{cap}_FK=0$.
	\end{enumerate}
\end{lemma}

\begin{proof}
	\begin{enumerate}[]
		\item (i) $\implies$ (ii): Since $K$ is polar, let $u$ be an $F$-superharmonic function in $B_{2r}$ such that $u|_K=\infty$. Recalling the definition of $F$-superharmonic functions, there exists a point $x_0 \in B_{2r} \setminus K$ such that $u(x_0) < \infty$. Since $u$ is lower semi-continuous and $u$ cannot attain the value $-\infty$, we may assume $\inf_{B_{2r}}u>-\infty$ by choosing a little smaller ball $B_{2r'}$ instead of $B_{2r}$. Then by adding a positive constant if necessary, we further assume $\inf_{B_{2r}}u \geq 0$, i.e. $u$ is non-negative in $B_{2r}$. Note that we still have $u$ is $F$-superharmonic in $B_{2r}$ and $u|_K=\infty$. Therefore, for any $\varepsilon>0$, we have $\varepsilon u \in \Phi_K^1(B_{2r})$ and so 
		\begin{align*}
		\hat{R}_K^1(B_{2r}) \leq \varepsilon u.
		\end{align*}
		Letting $\varepsilon \to 0$ and taking $x=x_0$, we notice that $\hat{R}_K^1(B_{2r})(x_0)=0$. Finally, the strong minimum principle implies that $\mathrm{cap}_FK=0$.
		
		\item (ii) $\implies$ (i): Let $y_0=x_0+\frac{3}{2}re_1$. Then by the definition of the capacity and the capacity potential, we have $\hat{R}_K^1(B_{2r}(x_0))(y_0)=0$. Thus, there exists a sequence of $F$-superharmonic functions $\{u_j\}_{j=1}^{\infty}$ in $B_{2r}$ such that
		\begin{align*}
		\text{$u_j \geq 0$ in $B_{2r}$, $u_j \geq 1$ on $K$ and $u_j(y_0) <1/2^j$}.
		\end{align*}
		Define $v_k:=\sum_{j=1}^ku_j$ which is lower semi-continuous and is finite in a dense subset of $\Omega$. Furthermore, since $F$ is convex, we have $F(D^2v_k) \leq 0$, and so $v_k$ is $F$-superharmonic. Since $\{v_k\}_{k}$ is an increasing sequence of $F$-superharmonic functions, Lemma \ref{easy} (iii) gives that the limit function $v=v_k$ is either $F$-superharmonic or $v \equiv \infty$. The second possibility is excluded because $0 \leq v(y_0) \leq 1$. Therefore, $v$ is $F$-superharmonic in $B_{2r}$ and $v|_{K}=\infty$, which implies that 
		$K$ is polar.
	\end{enumerate}
\end{proof}

\begin{definition} [Removable sets] 
	A compact set $K (\subset B_{7r/5})$ is called \textit{$F$-removable}, or simply \textit{removable}, if for each function $u$ that is $F$-superharmonic on $B_{2r} \setminus K$ and is bounded below in a neighborhood of $K$, there exists an extension $U$ of $u$ which is $F$-superharmonic in $B_{2r}$ and $U=u$ in $B_{2r} \setminus K$. 
\end{definition}

\begin{lemma}
	Suppose that $K$ is a compact set of capacity zero and $F$ is convex. Then $K$ is removable.
\end{lemma}

\begin{proof}
	Let $u$ be an $F$-superharmonic function in $B_{2r} \setminus K$ and is bounded below in a neighborhood of $K$. Since $K$ is of capacity zero, we have $\hat{R}^1_K(B_{2r}) (y_0)=0$ and so $\hat{R}_K^1(B_{2r}) \equiv 0$ by the strong minimum principle. In particular, $R_K^1(B_{2r}) \equiv 0$ in $B_{2r} \setminus K$. Now, for any $z_0 \in B_{2r} \setminus K$, following the proof of [(ii) $\implies$ (i)] part in Lemma \ref{polar}, there exists a non-negative $F$-superharmonic function $v_{z_0}$ in $B_{2r}$ such that $v_{z_0}|_K=\infty$ and $v_{z_0}(z_0) <\infty$. 
	
	Now we consider a canonical lower semi-continuous extension $U$ of $u$ across $K$, which is defined by
	\begin{align*}
	U(x)=		
	\left\{ \begin{array}{ll} 
	\liminf_{y \to x, y \not\in K}u(y) & \textrm{if $x \not \in \mathrm{int}K$},\\
	\infty & \textrm{if $x \in \mathrm{int}K$}.
	\end{array} \right.
	\end{align*}
	Then $U$ is the lower semi-continuous regularization of the function $v$, where
	\begin{align*}
	v=
	\left\{ \begin{array}{ll} 
	u & \textrm{in $B_{2r} \setminus K$},\\
	\infty & \textrm{on $K$}.
	\end{array} \right.
	\end{align*}
	See \cite{HL14} for details. Moreover, by Lemma \ref{lsc} and Lemma \ref{equiv}, we notice that $U=u$ in $B_{2r} \setminus K$ and so $U$ is $F$-superharmonic in $B_{2r} \setminus K$.
	
	Then we claim that $U+\varepsilon v_{z_0}$ is $F$-superharmonic in $B_{2r}$, for any $\varepsilon>0$ and $z_0 \in B_{2r} \setminus K$. Indeed, the convexity of $F$ immediately guarantees that $U+\varepsilon v_{z_0}$ is $F$-superharmonic in $B_{2r} \setminus K$. On the other hand, since $U+\varepsilon v_{z_0}|_K=\infty$, we cannot choose any test functions for $U+\varepsilon v_{z_0}$ at points in $K$. In other words, for any $\varphi \in C^2(\Omega)$, $U+\varepsilon v_{z_0}-\varphi$ cannot have a local minimum at $x_0 \in K$. Thus, recalling the equivalence of $F$-supersolution and $F$-superharmonic function (Theorem \ref{equiv}), we conclude that $U+\varepsilon v_{z_0}$ is $F$-superharmonic in $B_{2r}$.
	
	Now let $\mathcal{F}=\{U+\varepsilon v_{z_0}\}_{\varepsilon>0, z_0 \in B_{2r} \setminus K}$ be a family of $F$-superharmonic functions in $B_{2r}$. Since $u$ is bounded below in a neighborhood of $K$ and $v_{z_0}$ is non-negative, any element in $\mathcal{F}$ is locally uniformly bounded below. Thus, applying Lemma \ref{balayagesuper}, we have
	\begin{align*}
	s(x)=\lim_{r \to 0}\inf_{B_r(x)} (\inf \mathcal{F})
	\end{align*}
	is $F$-superharmonic in $B_{2r}$. On the other hand, it is easy to check that
	\begin{align*}
	\inf \mathcal{F}=
	\left\{ \begin{array}{ll} 
	u & \textrm{in $B_{2r} \setminus K$},\\
	\infty & \textrm{on $K$}.
	\end{array} \right.
	\end{align*}
	Therefore, we conclude that $s=U$ and $U$ is a desired extension of $u$.
\end{proof}

\begin{remark}
	Considering the dual operator $\widetilde{F}$, one can obtain analogous definitions and corresponding results when the operator is concave.
	
	For similar results concerning polar sets and removable sets, see \cite{HKM93} for $p$-Laplacian operators, \cite{Lab02b} for Pucci extremal operators, and \cite{Lab02a} for $k$-Hessian operators. See also \cite{AGV13, HL13, HL14} for the analysis of polar sets and removable sets in view of Riesz capacity or Hausdorff measure.
\end{remark}

\subsection{Another characterization of a regular point}
The definitions of a reduced function and a balayage depend on the choice of an operator $F$. In this subsection, we need to distinguish an operator and its dual operator, so we will specify the dependence by denoting
	$\hat{R}_{K}^{1, F}(\Omega)$ or $\hat{R}_K^{1, \widetilde{F}}(\Omega)$.
We now provide a key lemma for our first main theorem, the sufficiency of the Wiener criterion:
\begin{lemma} \label{keylem}
	A boundary point $x_0 \in \partial \Omega$ is regular if
	\begin{align*}
	\hat{R}_{\overline{B}\setminus \Omega}^{1, \widetilde{F}}(2B)(x_0)=1=\hat{R}_{\overline{B}\setminus \Omega}^{1, {F}}(2B)(x_0)
	\end{align*}
	whenever $B$ is a ball centered at $x_0$.
\end{lemma}

\begin{proof}
	For $f \in C(\partial \Omega)$, consider the upper Perron solution $\overline{H}_f=\overline{H}_f(\Omega)$. We may assume $f(x_0)=0$ and $\max_{\partial \Omega}|f| \leq 1$. For $\varepsilon>0$, we can choose a ball $B$ with center $x_0$ such that $\partial (2B) \cap \Omega \neq\varnothing$ and $|f|<\varepsilon$ in $2B \cap \partial \Omega$. Then we define
	\begin{align*}
	u=		
	\left\{ \begin{array}{ll} 
	1+\varepsilon-\hat{R}_{\overline{B} \setminus \Omega}^{1, \widetilde{F}}(2B) & \textrm{in $\Omega \cap 2B$},\\
	1+\varepsilon & \textrm{in $\Omega \setminus 2B$}.
	\end{array} \right.
	\end{align*}
	Since $\hat{R}_{\overline{B}\setminus \Omega}^{1, \widetilde{F}}(2B)$ is a $\widetilde{F}$-solution in $\Omega \cap 2B$, $1+\varepsilon-\hat{R}_{\overline{B}\setminus \Omega}^{1, \widetilde{F}}(2B)$ is $F$-harmonic in $\Omega \cap 2B$. On the other hand, by Theorem \ref{cappoten}, $\hat{R}_{\overline{B}\setminus \Omega}^{1, \widetilde{F}}(2B)$ can be considered as the upper Perron solution for the operator $\widetilde{F}$. Then since a ball is regular, we have
	\begin{align*}
	\lim_{y \to x}\hat{R}_{\overline{B}\setminus \Omega}^{1, \widetilde{F}}(2B)(y)=0 \quad \text{for all $x \in \partial (2B)$}.
	\end{align*}
	Thus, $u$ is continuous in $\Omega$ and by the pasting lemma, $u$ is $F$-superharmonic in $\Omega$. Moreover, it can be easily checked that 
	\begin{align*}
	\liminf_{y \to x}u(y) \geq f(x) \quad \text{for any $x\in \partial \Omega$.}
	\end{align*}
	Therefore, $u \in \mathcal{U}_f$ and so $\overline{H}_f \leq u$. In particular,
	\begin{align*}
	\limsup_{\Omega \ni y \to x_0}\overline{H}_f(y) \leq \limsup_{\Omega \ni y \to x_0}u(y)=1+\varepsilon-\liminf_{\Omega \ni y \to x_0}\hat{R}_{\overline{B}\setminus \Omega}^{1, \widetilde{F}}(2B)(y) \leq 1+\varepsilon-\hat{R}_{\overline{B}\setminus \Omega}^{1, \widetilde{F}}(2B)(x_0)=\varepsilon.
	\end{align*}
	For the converse inequality, we define
	\begin{align*}
	v=		
	\left\{ \begin{array}{ll} 
	-1-\varepsilon+\hat{R}_{\overline{B} \setminus \Omega}^{1, {F}}(2B) & \textrm{in $\Omega \cap 2B$},\\
	-1-\varepsilon & \textrm{in $\Omega \setminus 2B$}.
	\end{array} \right.
	\end{align*}
	Then by a similar argument, $v \in \mathcal{L}_f$ and so, 
	\begin{align*}
	\liminf_{\Omega \ni y \to x_0}\underline{H}_f(y) \geq -\varepsilon.
	\end{align*}
	Consequently, since $\varepsilon>0$ is arbitrary, we conclude that
	\begin{align*}
	\lim_{\Omega \ni y \to x_0}\overline{H}_f(y)=\lim_{\Omega \ni y \to x_0}\underline{H}_f(y)=0=f(x_0),
	\end{align*}
	i.e. $x_0$ is regular.
\end{proof}


Next, we exhibit a converse direction of the above lemma: i.e. a characterization of an irregular boundary point. We expect that this lemma may be employed to prove the necessity of the Wiener criterion for the general case.
\begin{lemma}[Characterization of an irregular boundary point] \label{irrlem}
	If there exists a constant $\rho>0$ such that the capacity potential $u=u_{\rho}$ of ${\overline{B_{\rho}(x_0)} \setminus \Omega}$ with respect to $B_{2\rho}(x_0)$ satisfies the inequality
	\begin{align*}
	u(x_0)=\hat{R}_{\overline{B_{\rho}(x_0)} \setminus \Omega}^1(B_{2\rho}(x_0))<1,
	\end{align*}
	then the boundary point $x_0 \in \partial \Omega$ is irregular.
\end{lemma}

\begin{proof}
	Since the capacity potential $u$ is the lower semi-continuous regularization, we have
	\begin{align} \label{irr}
	u(x_0)=\liminf_{\Omega \ni x \to x_0}u(x)<1.
	\end{align}
	Moreover, by definition, we have $u_{\rho'} \leq u_{\rho}$ when $0<\rho'<\rho$. Thus, we can choose a sufficiently small $\rho>0$ such that \eqref{irr} holds and $\Omega \cap \partial B_{2\rho}(x_0)\neq \varnothing$. 
	
	Now we define a smooth boundary data $f$ on $\partial (\Omega \cap B_{2\rho}(x_0))$ such that
	$f(x)=3/2$ if $x \in \partial \Omega \cap B_{\rho/2}(x_0)$, $0 \leq f(x) \leq 3/2$ if $x \in \partial \Omega \cap (B_{\rho}(x_0) \setminus B_{\rho/2}(x_0))$ and $f(x)=0$ on the remaining part of $\partial (\Omega \cap B_{2\rho}(x_0))$. Then we consider the lower Perron solution $\underline{H}_{f}(\Omega \cap B_{2\rho}(x_0))$. We claim that the following inequality holds:
	\begin{align} \label{irr2}
	\underline{H}_{f}(x) \leq \frac{1}{2}+u(x), \quad x \in \Omega \cap B_{2\rho}(x_0).
	\end{align}
	Recalling the comparison principle, it is enough to check the above inequality on the boundary of the domain $\Omega \cap B_{2\rho}(x_0)$. For this purpose, let $v \in \mathcal{L}_{f}(\Omega \cap B_{2\rho}(x_0))$ and $w \in \mathcal{U}_{g} (B_{2\rho}(x_0) \setminus (\overline{B_{\rho}(x_0)} \setminus \Omega))$ where $g$ is given by  (recall Theorem \ref{cappoten})
	\begin{align*}
	g=	
	\left\{ \begin{array}{ll} 
	1 & \text{on $\partial(\overline{B_{\rho}(x_0)} \setminus \Omega)$},\\
	0 & \text{in $\partial B_{2\rho}(x_0)$}.
	\end{array} \right.
	\end{align*}
	\begin{enumerate}[(i)]
		\item (on $\partial \Omega \cap B_{2\rho}(x_0)$) 
		First, for $x \in \partial \Omega \cap B_{\rho}(x_0)$, we have
		\begin{align*}
		\limsup_{y \to x}v(y) \leq f(x) \leq \frac{3}{2} =\frac{1}{2}+g(x) \leq \frac{1}{2}+\liminf_{y \to x}w(y).
		\end{align*}
		Next, for $x \in \partial \Omega \cap (B_{2\rho}(x_0) \setminus B_{\rho}(x_0))$, we have
		\begin{align*}
		\limsup_{y \to x}v(y) \leq f(x)=0 \leq \frac{1}{2}+g(x) \leq \frac{1}{2}+\liminf_{y \to x}w(y).
		\end{align*}
		
		\item (on $\Omega \cap \partial B_{2\rho}(x_0)$) Similarly, we obtain
		\begin{align*}
		\limsup_{y \to x}v(y) \leq f(x)=0 \leq \frac{1}{2}+g(x) \leq \frac{1}{2}+\liminf_{y \to x}w(y).
		\end{align*}
	\end{enumerate}
	Now since $v$ and $w$ are $F$-subharmonic and $F$-superharmonic, respectively, we derive that
	\begin{align*}
	v \leq \frac{1}{2}+w, \quad \text{in $\Omega \cap B_{2\rho}(x_0)$}.
	\end{align*}
	Taking the supremum on $v$ and the infimum on $w$, we conclude \eqref{irr2} which implies that
	\begin{align*}
	\liminf_{\Omega \cap B_{2\rho}(x_0)\ni x \to x_0}\underline{H}_f(x) \leq \frac{1}{2}+\liminf_{\Omega \cap B_{2\rho}(x_0)\ni x \to x_0}u(x) <\frac{3}{2}=f(x_0).
	\end{align*}
	Therefore, $x_0$ is irregular with respect to  $\Omega \cap B_{2\rho}(x_0)$. Recalling Lemma \ref{local}, we deduce that $x_0$ is irregular with respect to $\Omega$.
\end{proof}

\section{A sufficient condition for the regularity of a boundary point}
In this section, we prove the sufficiency of the Wiener criterion and its sequential corollaries, via the potential estimates. More precisely, we first develop quantitative estimates for the capacity potential $\hat{R}_K^1(B)$ by employing capacitary estimates obtained in Section 3. Then we adopt the characterization of a regular boundary point in terms of the capacity potential to deduce the desired conclusion.

\begin{definition}
	We say that a set $E$ is \textit{$F$-thick} at $z$ if the Wiener integral diverges, i.e.
	\begin{align}\label{wieint}
	\int_0^1 \mathrm{cap}_F (E \cap \overline{B_t(z)}, B_{2t}(z))  \frac{\mathrm{d}t}{t}=\infty.
	\end{align}
	For simplicity, we write
	\begin{align*}
	\varphi_F(z, E, t)=\mathrm{cap}_F (E \cap \overline{B_t(z)}, B_{2t}(z)),
	\end{align*}
	for the capacity density function in \eqref{wieint}.
\end{definition}

\begin{remark}
	Recalling Lemma \ref{capball}, there exists a constant $c>0$ which is independent of $t>0$ such that
	\begin{align*}
	1/c \leq \mathrm{cap}_F(\overline{B_t}, B_{2t}) \leq c.
	\end{align*}
	Thus, one may write an equivalent form of \eqref{wieint}:
	\begin{align*}
	\int_0^1 \frac{\mathrm{cap}_F (E \cap \overline{B_t(z)}, B_{2t}(z))}{\mathrm{cap}_F (\overline{B_t(z)}, B_{2t}(z))}  \frac{\mathrm{d}t}{t}=\infty,
	\end{align*}
	which is a similar form to the Wiener integral appearing in \cite{KM94, Wie24a}.
\end{remark}

Now we can state an equivalent form of our main theorem, Theorem \ref{wiener}:
\begin{align*}
\text{If $\Omega^c$ is both $F$-thick and $\widetilde{F}$-thick at a boundary point $x_0 \in \partial \Omega$, then $x_0$ is regular.}
\end{align*}
To prove this statement, we need several auxiliary lemmas regarding the capacity potential.

\begin{lemma}\label{est0}
	Fix a ball $B$. Suppose that $K \subset B'$ is compact and $v=\hat{R}_K^1(B)$. If $0<\gamma<1$ and $K_{\gamma}:=\{x \in B : v(x) \geq \gamma\} \subset B'$, then
	\begin{align*}
	\mathrm{cap} (K_{\gamma}, B)=\frac{1}{\gamma}\mathrm{cap} (K, B).
	\end{align*}
\end{lemma}

\begin{proof}
	We write $v_{\gamma}:=\hat{R}_{K_{\gamma}}^1(B)$. Then by Lemma \ref{balayage} and the definition of a reduced function, 
	\begin{align*}
	v_{\gamma}=R_{K_{\gamma}}^1(B)=\inf \Phi_{K_{\gamma}}^1=\inf\{\text{$w$: $w$ is $F$-superharmonic in $B$, $w \geq \psi_{\gamma}$ in $B$}\} \quad \text{in $B \setminus K_{\gamma}$},
	\end{align*}
	where 
	\begin{align*}
	\psi_{\gamma}=
	\left\{ \begin{array}{ll} 
	1 & \textrm{in $K_{\gamma}$},\\
	0 & \textrm{in $B \setminus K_{\gamma}$}.
	\end{array} \right.
	\end{align*}
	
	\begin{enumerate}[(i)]
		\item Clearly, $v=\hat{R}_{K}^1(B)$ is $F$-superharmonic in $B$ and so is $v/\gamma$ due to (F2). Since $v \geq \gamma$ in $K_{\gamma}$, we have $v/\gamma \geq 1$ in $K_{\gamma}$. Thus, $v/\gamma \in \Phi_{K_{\gamma}}^1$ and so 
		\begin{align*}
		\frac{v}{\gamma} \geq v_{\gamma} \quad \text{in $B \setminus K_{\gamma}$}.
		\end{align*}
		
		\item Recalling Theorem \ref{cappoten}, $v_{\gamma}=\overline{H}_{f_{\gamma}}(B \setminus K_{\gamma})$ in $B \setminus K_{\gamma}$ where
		\begin{align*}
		f_{\gamma}=
		\left\{ \begin{array}{ll} 
		1 & \textrm{on $\partial K_{\gamma}$},\\
		0 & \textrm{on $\partial B$}.
		\end{array} \right.
		\end{align*}
		Then for $u \in \mathcal{U}_{f_{\gamma}}(B \setminus K_{\gamma})$, we have
		\begin{align*}
		\liminf_{B \setminus K_{\gamma} \ni y \to x}u(y) \geq f_{\gamma}(x)=1=\frac{v(x)}{\gamma},
		\end{align*}
		for any $x \in \partial K_{\gamma}$. Since $u$ is $F$-superharmonic and $v/\gamma$ is $F$-harmonic in $B \setminus K_{\gamma}$, the comparison principle leads to $u \geq v/\gamma$ in $B \setminus K_{\gamma}$ and so
		\begin{align*}
		v_{\gamma} \geq \frac{v}{\gamma}  \quad \text{in $B \setminus K_{\gamma}$}.
		\end{align*}
	\end{enumerate}
	Consequently, we conclude that
	\begin{align*}
	\mathrm{cap} (K_{\gamma}, B)=v_{\gamma}(y_0)=\frac{1}{\gamma}v(y_0)=\frac{1}{\gamma}\mathrm{cap} (K,B).
	\end{align*}
\end{proof}

\begin{lemma} \label{est}
	Fix a ball $B=B_{2r}(x_0)$.
	 Let $K \subset B_r=B_r(x_0)$ be a compact set and $v=\hat{R}_K^1(B)$. Then there exists a constant $c>0$ which is independent of $K$ and $r$ such that
	 \begin{align*}
		v(x) \geq c \, {\mathrm{cap} (K,B)},
	 \end{align*}
	 for any $x \in B_r$.
\end{lemma}

\begin{proof}
	Denote 
	\begin{align*}
	M:=\sup_{\partial B_{6r/5}} v, \quad m:=\inf_{\partial B_{6r/5}} v.
	\end{align*}
	Since $v$ is a non-negative $F$-solution in $B \setminus K$, Harnack inequality yields that there exists a constant $c_1>0$ independent of $r>0$ such that
	\begin{align} \label{aux1}
		c_1M \leq m.
	\end{align}
	Morevoer, the strong maximum principle in $B \setminus B_{6r/5}$ implies that 
	\begin{align*}
	K_M:=\{v \in B: v(x) \geq M\} \subset \overline{B_{6r/5}},
	\end{align*}
	and so
	\begin{align}\label{aux2}
	\mathrm{cap} (K_M, B) \leq \mathrm{cap} (\overline{B_{6r/5}},B) \sim 1.
	\end{align}
	Here we applied Lemma \ref{capball} and the comparable constant does not depend on $K$ and $r$. 
	
	Now since $K_M \subset B'$, we can apply Lemma \ref{est0}:
	\begin{align} \label{aux3}
	\mathrm{cap} (K_M, B)=\frac{1}{M} \mathrm{cap} (K, B).
	\end{align}
	Finally, combining \eqref{aux1}, \eqref{aux2} and \eqref{aux3}, we conclude that
	\begin{align*}
	m \geq c_1 M = c_1 \cdot \frac{\mathrm{cap} (K, B)}{\mathrm{cap} (K_M, B)} \geq c_2  \, {\mathrm{cap} (K,B)},
	\end{align*}
	and the minimum principle leads to the desired result.
\end{proof}

We may rewrite the previous lemma as
\begin{align} \label{impest}
 \hat{R}_K^1(B_{2r}) (x) \geq c\, \varphi_F(x_0, K, r), \quad \text{for any $x \in B_r$}.
\end{align}

\begin{lemma} \label{potential}
	Let $x_0 \in \partial \Omega$, $\rho>0$ and 
	\begin{align*}
	w=1-\hat{R}_{\overline{B_{\rho}(x_0)} \setminus \Omega}^{1}( B_{2\rho}(x_0)).
	\end{align*}
	Then for all $0<r \leq \rho$, there exists a constant $c>0$ such that
	\begin{align*}
	w(x) \leq \exp\Big(-c\int_r^{\rho}\varphi_F(x_0, \Omega^c, t)\frac{\mathrm{d}t}{t} \Big),
	\end{align*}
	for any $x \in B_r(x_0)$.
\end{lemma}

\begin{proof}
	Denote $B_i=B_{2^{1-i}\rho}(x_0)$. Fix $0<r\leq \rho$ and let $k$ be the integer with $2^{-k}\rho < r \leq 2^{1-k}\rho$. Then write for $i=0, 1,2,...$
	\begin{align*}
	v_i:=\hat{R}_{\overline{B_{i+1}} \setminus \Omega}^{1}(B_i)
	\end{align*}
	and
	\begin{align*}
	a_i:=\varphi_{F}(x_0, \Omega^c, 2^{-i}\rho).
	\end{align*}
	Since $e^t \geq 1+t$, estimate \eqref{impest} yields that
	\begin{align*}
	v_i \geq c a_i \geq 1-\exp(-ca_i) \quad \text{in $B_{i+1}$}.
	\end{align*}
	Thus, denoting $m_0:=\inf_{B_1} v_0$, we have
	\begin{align*}
	1-m_0 \leq \exp(-ca_0).
	\end{align*}
	Next, let $D_1:=B_1 \setminus (\overline{B_2} \cap \Omega^c)$ and 
	\begin{align*}
		\psi_1:=	
		\left\{ \begin{array}{ll} 
		1 & \text{in $\overline{B_2} \cap \Omega^c$},\\
		m_0 & \text{in $D_1$}.
		\end{array} \right.
	\end{align*}
	Then we write $u_1:=\hat{R}^{\psi_1}(B_1)$ be the balayage with respect to the $\psi_1$ in $B_1$. It immediately follows from the definition of balayage that
	\begin{align*}
	\frac{u_1-m_0}{1-m_0}=\hat{R}_{\overline{B_2} \cap \Omega^c}^1(B_1)=v_1.
	\end{align*}
	Again, denoting $m_1:=\inf_{B_2}u_1$, we obtain
	\begin{align*}
	1-m_1 \leq (1-m_0) \exp(-ca_1) \leq \exp(-c(a_1+a_0)).
	\end{align*}
	Now iterate this step: let $D_i:=B_i \setminus (\overline{B_{i+1}} \cap \Omega^c)$ and 
	\begin{align*}
	\psi_i=	
	\left\{ \begin{array}{ll} 
	1 & \text{in $\overline{B_{i+1}} \cap \Omega^c$},\\
	m_{i-1} & \text{in $D_i$}.
	\end{array} \right.
	\end{align*}
	Denoting $u_i:=\hat{R}^{\psi_i}(B_i)$ and $m_{i}:=\inf_{B_{i+1}}u_i$, we have 
	\begin{align*}
	\frac{u_i-m_{i-1}}{1-m_{i-1}}=v_i
	\end{align*}
	and so
	\begin{align*}
	1-m_i \leq (1-m_{i-1}) \exp(-ca_i) \leq \exp\Big(-c \sum_{j=0}^i a_j\Big).
	\end{align*}
	Furthermore, we claim that $u_i \geq u_{i+1}$ in $B_{i+1}$. Indeed, by Theorem \ref{cappoten}, $u_i=\overline{H}_{f_i}(D_i)$ in $D_i$ where $f_i \in C(\partial D_i)$ is given by 
	\begin{align*}
	f_i=
	\left\{ \begin{array}{ll} 
	1 & \text{in $\partial (\overline{B_{i+1}} \cap \Omega^c)$},\\
	m_{i-1} & \text{in $\partial B_i$}.
	\end{array} \right.
	\end{align*}
	Thus, for $u \in \mathcal{U}_{f_i}(D_i)$, we have
	\begin{align*}
	&\liminf_{D_{i+1} \ni y \to x}u(y) \geq 1 \geq \limsup_{D_{i+1} \ni y \to x} u_{i+1}(y) \quad \text{for any $x \in \partial (\overline{B_{i+2}} \cap \Omega^c)$},\\
	&\liminf_{D_{i+1} \ni y \to x}u(y) \geq m_i = \limsup_{D_{i+1} \ni y \to x} u_{i+1}(y) \quad \text{for any $x \in \partial B_{i+1}$}.
	\end{align*}
	Therefore, by the comparison principle, $u \geq u_{i+1}$ in $D_{i+1}$ and so $u_i=\overline{H}_{f_i}(D_i) \geq u_{i+1}$ in $B_{i+1}$. 
	
	Repeating the argument above, we conclude that $v_0 \geq u_1 \geq \cdots \geq u_k$ in $B_k$, which implies that
	\begin{align*}
	w=1-v_0 \leq 1-u_k \leq 1-m_k \leq \exp \Big(-c \sum_{j=0}^k a_j\Big) \quad \text{in $B_{k+1}$.}
	\end{align*}
	Finally, the result follows from
	\begin{align*}
	\int_r^{\rho} \varphi_F(x_0, \Omega^c, t)\frac{\mathrm{d}t}{t} \leq c \sum_{i=1}^k a_i,
	\end{align*}
	which can be easily checked from the dyadic decomposition. Indeed, we can deduce from Lemma \ref{capball} and Lemma \ref{capest} that if $t \leq s \leq 2t$, then 
	\begin{align*}
	\mathrm{cap}_F(\overline{B_t} \setminus K, B_{2t}) \sim \mathrm{cap}_F(\overline{B_t} \setminus K, B_{2s}),
	\end{align*}
	where the comparable constant only depends on $n, \lambda, \Lambda$ and these results also hold for $\mathrm{cap}_{\widetilde{F}}(\cdot)$.
\end{proof}

Now we are ready to prove the sufficiency of the Wiener criterion, Theorem \ref{wiener}.

\begin{proof}[Proof of Theorem \ref{wiener}]
	Let $x_0 \in \partial \Omega$, $\rho>0$ and define
	\begin{align*}
	w_{F, \rho}:=1-\hat{R}_{\overline{B_{\rho}(x_0)} \setminus \Omega}^{1, F}( B_{2\rho}(x_0)) \quad \text{and} \quad 
		w_{\widetilde{F}, \rho}:=1-\hat{R}_{\overline{B_{\rho}(x_0)} \setminus \Omega}^{1, \widetilde{F}}( B_{2\rho}(x_0)).
	\end{align*}
	Then applying Lemma \ref{potential} for both functions, we have that for all $0<r \leq \rho$, there exist a constant $c_1, c_2>0$ such that
	\begin{align*}
	&w_{F, \rho}(x) \leq \exp\Big(-c_1\int_r^{\rho}\varphi_F(x_0, \Omega^c, t)\frac{\mathrm{d}t}{t} \Big),\\
	&w_{\widetilde{F}, \rho}(x) \leq \exp\Big(-c_2\int_r^{\rho}\varphi_{\widetilde{F}}(x_0, \Omega^c, t)\frac{\mathrm{d}t}{t} \Big),
	\end{align*}
	for any $x \in B_r(x_0)$. Letting $r \to 0^+$, we conclude that
	\begin{align*}
	\hat{R}_{\overline{B_{\rho}(x_0)} \setminus \Omega}^{1, F}( B_{2\rho}(x_0)) (x_0)=1=	\hat{R}_{\overline{B_{\rho}(x_0)} \setminus \Omega}^{1, \widetilde{F}}( B_{2\rho}(x_0)) (x_0).
	\end{align*}
	Since $\rho>0$ can be arbitrarily chosen, an application of Lemma \ref{keylem} yields that $x_0 \in \partial \Omega$ is a regular boundary point. (Note that a boundary point $x_0$ is $F$-regular if and only if it is $\widetilde{F}$-regular; Corollary \ref{ftildef}.)
\end{proof}

On the other hand, if additional information is imposed on the boundary data $f$, i.e. the boundary data $f$ has its maximum (or minimum) at $x_0 \in \partial \Omega$, then we can deduce the continuity of the Perron solution at $x_0$ under a relaxed condition:
\begin{corollary} \label{specialwie}
	Suppose that $f \in C(\partial \Omega)$ attains its maximum [resp. minimum] at $x_0 \in \partial \Omega$. If $\Omega^c$ is $F$-thick [resp. $\widetilde{F}$-thick] at $x_0 \in \partial \Omega$, then 
	\begin{align*}
	\lim_{\Omega \ni y \to x_0}\overline{H}_f(y)=f(x_0)=\lim_{\Omega \ni y \to x_0}\underline{H}_f(y).
	\end{align*}
\end{corollary}
\begin{proof}
	Similarly as in the proof of the previous theorem, this corollary is the consequence of Lemma \ref{keylem} and Lemma \ref{potential}.
\end{proof}

Furthermore, if the given boundary data $f \in C(\partial \Omega)$ is resolutive, then we are able to obtain a quantitative estimate for the modulus of continuity.

\begin{lemma}[The modulus of continuity]\label{modulus}
	Suppose that $\Omega$ is an open and bounded subset of $\mathbb{R}^n$. Let $f \in C(\partial \Omega)$. 
	
	If $x_0 \in \partial \Omega$ with $f(x_0)=0$, then for $0<r \leq \rho$, we have
	\begin{align*}
	\sup_{\Omega_r} \underline{H}_f^F \leq \max_{\partial \Omega_{2\rho}}f+\max_{\partial \Omega}f \cdot \exp\Big(-c\int_r^{\rho} \varphi_{\widetilde{F}}(x_0, \Omega^c, t)\frac{\mathrm{d}t}{t}\Big)
	\end{align*}
	and
	\begin{align*}
	\inf_{\Omega_r} \overline{H}_f^F \geq \min_{\partial \Omega_{2\rho}}f+\min_{\partial \Omega}f \cdot \exp\Big(-c\int_r^{\rho} \varphi_{F}(x_0, \Omega^c, t)\frac{\mathrm{d}t}{t}\Big)
	\end{align*}
	where $\Omega_r:=\Omega \cap B_r(x_0)$ and $\partial \Omega_{2\rho}:=\partial \Omega \cap B_{2\rho}(x_0)$.
	
	Furthermore, if $f$ is resolutive, then we have the quantitative estimate:
	\begin{align*}
	\min_{\partial \Omega_{2\rho}}f&+\min_{\partial \Omega}f \cdot \exp\Big(-c\int_r^{\rho} \varphi_{F}(x_0, \Omega^c, t)\frac{\mathrm{d}t}{t}\Big) \leq \inf_{\Omega_r} {H}_f^F\\
	 &\leq \sup_{\Omega_r}{H}_f^F \leq \max_{\partial \Omega_{2\rho}}f+\max_{\partial \Omega}f \cdot \exp\Big(-c\int_r^{\rho} \varphi_{\widetilde{F}}(x_0, \Omega^c, t)\frac{\mathrm{d}t}{t}\Big),
	\end{align*}
	where $H_f^F:=\overline{H}_f^F=\underline{H}_f^F$. 
\end{lemma}

\begin{proof}
	 Let $v=\hat{R}_{\overline{B_{\rho}(x_0)} \setminus \Omega}^{1, \widetilde{F}}( B_{2\rho}(x_0))$ be the capacity potential of $\overline{B_{\rho}} \setminus \Omega$ with respect to $B_{2\rho}$. 
	 
	Then let $w:=1-v$ and write 
	\begin{align*}
	s:=w \cdot \max_{\partial \Omega}f+\max_{\partial \Omega_{2\rho}}f.
	\end{align*}
	Note that since we assumed $f(x_0)=0$, we have $\max_{\partial \Omega}f \geq 0$ and $\max_{\partial \Omega_{2\rho}}f \geq 0$. For $u \in \mathcal{L}_f^F$, $u$ is $F$-subharmonic and $s$ is $F$-harmonic in $\Omega_{2\rho}$.
	Moreover, 
	\begin{align*}
	\liminf_{y \to x}s(y) \geq \max_{\partial \Omega_{2\rho}}f \geq \limsup_{y \to x}u(y) \quad \text{for any $x \in \partial \Omega \cap B_{2\rho}$}
	\end{align*}
	and
	\begin{align*}
	\liminf_{y \to x}s(y) \geq \max_{\partial \Omega}f \geq \limsup_{y \to x}u(y) \quad \text{for any $x \in \Omega \cap \partial B_{2\rho}$}.
	\end{align*}
	 Thus, the comparison principle yields that $s \geq u$ in $\Omega_{2\rho}$ and so $s \geq \underline{H}_f^F$ in $\Omega_{2\rho}$. 
	 
	 On the other hand, let 
	 \begin{align*}
	 \widetilde{s}:=\Big(1-\hat{R}_{\overline{B_{\rho}(x_0)} \setminus \Omega}^{1, F}(B_{2\rho}(x_0))\Big) \cdot \max_{\partial \Omega}(-f)+\max_{\partial \Omega_{2\rho}}(-f).
	 \end{align*}
	 By the same argument, we derive $\widetilde{s} \geq \underline{H}_{-f}^{\widetilde{F}}=-\overline{H}_f^F$ in $\partial \Omega_{2\rho}$.
	 
 An application of Lemma \ref{potential} for $w$ (and $\widetilde{w}$) finishes the proof. 
\end{proof}

Now we present a new geometric condition for a regular boundary point, namely the exterior corkscrew condition; see also \cite{JK82, LZLH20}.
\begin{definition}
	We say that $\Omega$ satisfies the \textit{exterior corkscrew condition} at $x_0 \in \partial \Omega$ if there exists $0<\delta<1/4$ and $R>0$ such that for any $0<r<R$, there exists $y \in B_r(x_0)$ such that $\overline{B_{\delta r}(y)}\subset \Omega^c \cap B_r(x_0)$.
\end{definition}

Note that if $\Omega$ satisfies an exterior cone condition at $x_0 \in \partial \Omega$, then $\Omega$ satisfies an exterior corkscrew condition at $x_0$. Thus, the following corollary obtained from the (potential theoretic) Wiener criterion is a generalized result of Theorem \ref{cone}.

\begin{corollary}[Exterior corkscrew condition]\label{corkscrew}
	Suppose that $\Omega$ satisfies an exterior corkscrew condition at $x_0 \in \partial \Omega$. Then $x_0$ is a regular boundary point. Moreover, if $f$ is H\"older continuous at $x_0$ and is resolutive, then $H_f$ is H\"older continuous at $x_0$.
\end{corollary}

\begin{proof}
	A small modification of Lemma \ref{capball} and its proof, we have
	\begin{align*}
	\mathrm{cap}(\overline{B_{\delta r}(y)}, B_{2r}(x_0)) \sim 1, \quad \text{for $\delta \in (0, 1/4)$ and $\overline{B_{\delta r}(y)} \subset B_{2r}(x_0)$},
	\end{align*}
	where the comparable constant depends only on $n, \lambda, \Lambda$ and $\delta$. Thus, if $x_0$ satisfies an exterior corkscrew condition, then we have
	\begin{align*}
	\int_0^1 \mathrm{cap}_F(\overline{B_t(x_0)} \setminus \Omega, B_{2t}(x_0)) \frac{\mathrm{d}t}{t} \geq \int_0^1 \mathrm{cap}_F(\overline{B_{\delta t}(y)}, B_{2t}(x_0)) \frac{\mathrm{d}t}{t} \geq \frac{1}{c} \, \int_0^1 \frac{\mathrm{d}t}{t}=\infty,\\
	\int_0^1 \mathrm{cap}_{\widetilde{F}}(\overline{B_t(x_0)} \setminus \Omega, B_{2t}(x_0)) \frac{\mathrm{d}t}{t} \geq \int_0^1 \mathrm{cap}_{\widetilde{F}}(\overline{B_{\delta t}(y)}, B_{2t}(x_0)) \frac{\mathrm{d}t}{t} \geq \frac{1}{c} \, \int_0^1 \frac{\mathrm{d}t}{t}=\infty,
	\end{align*}
	and so $x_0$ is a regular boundary point by the Wiener criterion.
	
	Next, for the second statement, we may assume $f(x_0)=0$ by adding a constant for $f$, if necessary. Since $f$ is resolutive, we can apply the quantitative estimate obtained in Lemma \ref{modulus}:
	\begin{align*}
	\min_{\partial \Omega_{2\rho}}f&+\min_{\partial \Omega}f \cdot \exp\Big(-c\int_r^{\rho} \varphi_{F}(x_0, \Omega^c, t)\frac{\mathrm{d}t}{t}\Big) \leq \inf_{\Omega_r} {H}_f^F\\
	&\leq \sup_{\Omega_r}{H}_f^F \leq \max_{\partial \Omega_{2\rho}}f+\max_{\partial \Omega}f \cdot \exp\Big(-c\int_r^{\rho} \varphi_{\widetilde{F}}(x_0, \Omega^c, t)\frac{\mathrm{d}t}{t}\Big),
	\end{align*}
	Here 
	\begin{enumerate}[(i)]
		\item $f$ is H\"older continuous at $x_0$: there exists a constant $C>0$ such that $|f(x)|=|f(x)-f(x_0)| \leq C|x-x_0|^{\gamma} \leq C\rho^{\gamma}$ for $x \in \partial \Omega_{2\rho}.$
		
		\item $\Omega$ satisfies an exterior corkscrew condition at $x_0$:
		\begin{align*}
		\exp\Big(-c\int_r^{\rho} \varphi_{\widetilde{F}}(x_0, \Omega^c, t)\frac{\mathrm{d}t}{t}\Big) \leq \exp\Big(-c_1\int_r^{\rho} \frac{\mathrm{d}t}{t}\Big)=\Big(\frac{r}{\rho}\Big)^{c_1}.
		\end{align*} 
	\end{enumerate}
	Thus, choosing $\rho=r^{1/2}$, we conclude that the Perron solution $H_f$ is H\"older continuous at $x_0$.
\end{proof}

\begin{remark}[Example]
	In this example, we suppose $n=2$, $F=\mathcal{P}_{\lambda, \Lambda}^+$ with ellipticity constants $0<\lambda<\Lambda$. Then it immediately follows that 
	\begin{align*}
	\widetilde{F}=\mathcal{P}_{\lambda, \Lambda}^-, \quad \alpha^{\ast}(F)=(n-1)\frac{\lambda}{\Lambda}-1<0, \quad \alpha^{\ast}(\widetilde{F})=(n-1)\frac{\Lambda}{\lambda}-1>0.
	\end{align*}
	We consider a domain $\Omega=B_1(0)\setminus\{0\} \subset \mathbb{R}^2$ and its boundary point $0 \in \partial \Omega$.
	\begin{enumerate}[(i)]
		\item Since $\alpha^{\ast}(F)<0$, we know that a single point has non-zero capacity. More precisely, recalling the homogeneous solution for $F$ is given by 
		\begin{align*}
		V(x)=-|x|^{1-\frac{\lambda}{\Lambda}},
		\end{align*}
		there exists a constant $c=c(\lambda, \Lambda)>0$ such that
		\begin{align*}
		{\mathrm{cap}_{{F}} (\{0\}, B_{2t}(0))}=c.
		\end{align*}
		Therefore, we have
		\begin{align*}
		\int_0^{\rho} {\mathrm{cap}_{{F}} (\{0\}, B_{2t}(0))}\frac{\mathrm{d}t}{t} = c\int_0^{\rho} \frac{\mathrm{d}t}{t}=\infty.
		\end{align*}
		In other words, $\Omega^c$ is $F$-thick at $0$.
		
		\item On the other hand, since $\alpha^{\ast}(\widetilde{F}) >0$, we know that a single point is of capacity zero.
		Therefore, we have
		\begin{align*}
		\int_0^{\rho} \frac{\mathrm{cap}_{\widetilde{F}} (\{0\}, B_{2t}(0))}{\mathrm{cap}_{\widetilde{F}} (\overline{B_t(0)}, B_{2t}(0))}\frac{\mathrm{d}t}{t}=0.
		\end{align*}
		In other words, $\Omega^c$ is not $\widetilde{F}$-thick at $0$ and we cannot apply our Wiener's criterion.
		
		\item Let $f_1 \in C(\partial \Omega)$ is a boundary data given by
		\begin{align*}
		f_1(x)=
		\left\{ \begin{array}{ll} 
		1 & \text{if $x=0$},\\
		0 & \text{if $|x|=1$}.
		\end{array} \right.
		\end{align*}
		Then clearly the function $u(x)=1-|x|^{1-\frac{\lambda}{\Lambda}}=1-V(x)$ is the solution for this Dirichlet problem. In particular, in this case, we have
		$\overline{H}_{f_1}=\underline{H}_{f_1}$ (i.e. $f_1$ is resolutive) and 
		\begin{align*}
		\lim_{\Omega \ni x \to 0}H_{f_1}(x)=1=f_1(0).
		\end{align*}
		Alternatively, one can apply Corollary \ref{specialwie} to reach the same conclusion, since $f_1$ attains its maximum at $0$ and $\Omega^c$ is $F$-thick at $0$.
		
		\item Let $f_2 \in C(\partial \Omega)$ is a boundary data given by
		\begin{align*}
		f_2(x)=
		\left\{ \begin{array}{ll} 
		-1 & \text{if $x=0$},\\
		0 & \text{if $|x|=1$}.
		\end{array} \right.
		\end{align*}
			Then since the zero function belongs to $\mathcal{U}_{f_2}$, we have $\overline{H}_{f_2} \leq 0$. Moreover, since $\varepsilon(1-|x|^{-(\frac{\Lambda}{\lambda}-1)}) \in \mathcal{L}_{f_2}$ for any $\varepsilon>0$, we have $\underline{H}_{f_2} \geq \varepsilon(1-|x|^{-(\frac{\Lambda}{\lambda}-1)})$. Letting $\varepsilon \to 0$, we conclude $\underline{H}_{f_2} \geq 0$.
			
		Therefore, we deduce that $\overline{H}_{f_2}=\underline{H}_{f_2}=0$. Furthermore, it follows that 
		\begin{align*}
		\lim_{\Omega \ni x \to 0}H_{f_2}(x)=0 \neq -1=f_2(0),
		\end{align*}
		which implies that $0$ is an irregular boundary point for $\Omega$. 
	\end{enumerate}

\end{remark}

\section{A necessary condition for the regularity of a boundary point}
In this section, we provide the necessity of the Wiener criterion, under additional structure on the operator $F$. Indeed, our strategy is to employ the argument made in \cite{LM85} which proved the necessity of the $p$-Wiener criterion for $p$-Laplacian operator with $p>n-1$. Since the assumption $p>n-1$ was essentially imposed to ensure the capacity of a line segment is non-zero in \cite{LM85}, we begin with finding the corresponding assumptions in the fully nonlinear case. 
 
\begin{lemma}\label{haus1}
	Suppose that $F$ is convex and $\alpha^{\ast}(F)>s$ for some $s>0$. Let $K$ be a compact subset in $B_r (\subset \mathbb{R}^n)$ such that $\mathcal{H}^s(K)<\infty$, where $\mathcal{H}^s$ is the $s$-dimensional Hausdorff measure. Then
	\begin{align*}
	\mathrm{cap}_F(K)=0.
	\end{align*}
\end{lemma} 
 
\begin{proof}
	For any $\delta>0$, define
	\begin{align*}
	\mathcal{H}_{\delta}^s(K):=\inf\sum_{i}\,r_i^s,
	\end{align*}
	where the infimum is taken over all countable covers of $K$ by balls $B_i$ with diameter $r_i$ not exceeding $\delta.$ Then since $\sup_{\delta>0}\mathcal{H}_{\delta}^s(K)= \lim_{\delta \to 0}\mathcal{H}_{\delta}^s(K)=\mathcal{H}^s(K) < \infty$ and $K$ is compact, for each $\delta \in (0, r)$, there exist finitely many open balls $\{B_i=B_{r_i}(x_i)\}_{i=1}^N$ such that $r_i <\delta$, $\bigcup_{i=1}^N B_i \supset K$, and
	\begin{align}\label{haus}
	 \sum_{i=1}^N \, r_i^s \leq \mathcal{H}^s(K)+1<\infty.
	\end{align}
	Now we consider the homogeneous solution $V(x)=|x|^{-\alpha^{\ast}} V\Big(\frac{x}{|x|}\Big)$ of $F$. Here we may assume $\min_{|x|=1}V(x)=1$ by normalizing $V$. If we let $W_i(x):=r_i^{\alpha^{\ast}}V(x-x_i)$, then it immediately follows that $W_i$ is non-negative and $F$-superharmonic in $\mathbb{R}^n$, and $W_i(x) \geq 1$ on $B_i$. 
	
	Finally, we let $W:=\sum_{i=1}^NW_i (\geq 0)$. Since $F$ is convex, $W$ is $F$-superharmonic in $\mathbb{R}^n$. Moreover, $W \geq 1$ on $\bigcup_{i=1}^NB_i$ and in particular, $W \geq 1$ on $K$. Therefore, $W \in \Phi_K^1(B_{4r})$ and so
	\begin{align*}
		\mathrm{cap}_F(K, B_{4r}) \leq W(y_0)\leq r^{-\alpha^{\ast}}\max_{|x|=1}V(x) \cdot\sum_{i=1}^Nr_i^{\alpha^{\ast}} \leq r^{-\alpha^{\ast}}\max_{|x|=1}V(x) \cdot (\mathcal{H}^s(K)+1)\,\delta^{\alpha^{\ast}-s},
	\end{align*}
	where we used \eqref{haus} and $\alpha^{\ast}>s$. Letting $\delta \to 0$, we finish the proof.
\end{proof} 
 
Now we prove the partial converse statement of Lemma \ref{haus1}. Indeed, here we only consider the compact set $K$ is given by a line segment $L$, whose Hausdorff dimension is exactly $1$.
\begin{lemma}\label{haus2}
	Suppose that $F$ is concave and $\alpha^{\ast}(F)<1$. Let $L=\{x_0+s e: ar\leq s \leq br\}$ be a line segment in $B_{r}(x_0)$, where $e$ is an unit vector in $\mathbb{R}^n$ and $0<a<b<1$ are constants satisfying $b-a<\frac{1}{2}$. Then
	\begin{align*}
		\mathrm{cap}_F(L, B_{2r})>0.
	\end{align*}
\end{lemma}

\begin{proof}
	Note that since $L$ is a line segment, for any $\delta>0$, one can cover $L$ by open balls $B_i=B_{3\delta}(x_i)$, $1 \leq i \leq N(\delta)$ where $x_i \in L$, $|x_i-x_j| \geq 2\delta$ whenever $i \neq j$, and $N(\delta) \sim \frac{(b-a)r}{\delta}$. We write such cover by $K_{\delta}:=\bigcup_{i=1}^{N(\delta)}\overline{B_i}$. Recalling Lemma \ref{capacitable} and its proof, for any $\varepsilon>0$, there exist a sufficiently small $\delta>0$ and corresponding cover $K_{\delta}$ such that
	\begin{align*}
	\mathrm{cap}_{F}(K_{\delta}, B_{2r}) \leq \mathrm{cap}_F(L, B_{2r})+\varepsilon.
	\end{align*}  
	If we denote $\widetilde{B}_i:=B_{\delta}(x_i)$ and $\widetilde{K}_{\delta}=\bigcup_{i=1}^{N(\delta)}\overline{\widetilde{B}}_i$, then we have $\widetilde{B}_i$ are pairwise disjoint and
	\begin{align*}
	\mathrm{cap}_F(\widetilde{K}_{\delta}, B_{2r}) \leq \mathrm{cap}_F(L, B_{2r})+ \varepsilon.
	\end{align*}
	On the other hand, for simplicity, we suppose that the homogeneous solution $V$ is given by
	\begin{align*}
	V(x)=|x|^{-\alpha^{\ast}}
	\end{align*}
	and $\alpha^{\ast}(F) \in (0,1)$. Note that if $\alpha^{\ast}<0$, then a single point has a positive capacity (Lemma \ref{sinpoint}) and the result immediately follows. Other cases can be shown by similar argument as in Lemma \ref{haus1}. For each $i=1,2, \cdots, N(\delta)$, write
	\begin{align*}
	W_i(x):=\Big(\frac{|x-x_i|}{\delta}\Big)^{-\alpha^{\ast}} \quad \text{and} \quad W(x)=\sum_{i=1}^{N(\delta)}W_i(x).
	\end{align*}
	Since $F$ is concave, $W$ is $F$-subharmonic in $\mathbb{R}^n \setminus \bigcup_{i=1}^{N(\delta)}\{x_i\}$.
	\begin{enumerate}[(i)]
		\item (On $\partial \widetilde{K}_{\delta}$) For $y \in \partial \widetilde{K}_{\delta}$, let $y \in \partial B_i$ for some $i$. Then for $j \neq i$, we have
		\begin{align*}
		|y-x_j|\geq |x_i-x_j| -|y-x_i|=|x_i-x_j|-\delta,
		\end{align*}
		and so
		\begin{align*}
		W(y) \leq 2(1+2^{-\alpha^{\ast}}+\cdots+N(\delta)^{-\alpha^{\ast}}) \leq 2\Big(1+\int_2^{N(\delta)}\frac{1}{s^{\alpha^{\ast}}}\,\mathrm{d}s\Big) \leq c\, N(\delta)^{1-\alpha^{\ast}}.
		\end{align*}
		Here we used the condition $\alpha^{\ast}<1$.
		
		\item (On $\partial B_{2r}$) For $z \in \partial B_{2r}$, 
		\begin{align*}
		|z-x_i| \geq 2r-br=(2-b)r,
		\end{align*}
		and so
		\begin{align*}
		W(z) \leq \Big(\frac{(2-b)r}{\delta}\Big)^{-\alpha^{\ast}} \cdot N(\delta).
		\end{align*}

	\end{enumerate}
	Therefore, for 
	\begin{align*}
	\widetilde{W}(x):=\frac{W(x)-\Big(\frac{(2-b)r}{\delta}\Big)^{-\alpha^{\ast}} \cdot N(\delta)}{c\, N(\delta)^{1-\alpha^{\ast}}},
	\end{align*}
	we have 
	\begin{align*}
	\text{$\widetilde{W}$ is $F$-subharmonic in $B \setminus \widetilde{K}_{\delta}$, \quad $\widetilde{W} \leq 0$ on $\partial B_{2r}$,\quad and \quad $\widetilde{W} \leq 1$ on $\partial \widetilde{K}_{\delta}.$}
	\end{align*}
	Note that since $\widetilde{K}_{\delta}$ and $B_{2r}$ are regular domains, the capacity potential $\hat{R}_{\widetilde{K}_{\delta}}^1(B_{2r})$ satisfies: 
	\begin{align*}
	\text{$\hat{R}_{\widetilde{K}_{\delta}}^1(B_{2r}) = 0$ on $\partial B_{2r}$,\quad and \quad $\hat{R}_{\widetilde{K}_{\delta}}^1(B_{2r})= 1$ on $\partial \widetilde{K}_{\delta}.$}
	\end{align*} 
	Hence, the comparison principle yields that 
	\begin{align*}
	\hat{R}_{\widetilde{K}_{\delta}}^1(B_{2r}) \geq \widetilde{W} \quad \text{in $B_{2r} \setminus \widetilde{K}_{\delta}.$}
	\end{align*}
	In particular, putting $x=x_0+\frac{3}{2}re$, we conclude that
	\begin{align*}
	|x-x_i| \leq 3r/2-ar=\Big(\frac{3}{2}-a\Big)r,
	\end{align*}
	and so
	\begin{align*}
	\hat{R}_{\widetilde{K}_{\delta}}^1(B_{2r}) \Big(x_0+\frac{3}{2}re\Big) \geq \widetilde{W}\Big(x_0+\frac{3}{2}re\Big) &\geq \frac{\Big[\Big(\frac{(3/2-a)r}{\delta}\Big)^{-\alpha^{\ast}}-\Big(\frac{(2-b)r}{\delta}\Big)^{-\alpha^{\ast}}\Big] \cdot N(\delta)}{c\, N(\delta)^{1-\alpha^{\ast}}}\\
	& \geq c_1 (b-a)^{\alpha^{\ast}} \Big[\Big(\frac{3}{2}-a\Big)^{-\alpha^{\ast}}-\Big(2-b\Big)^{-\alpha^{-\ast}}\Big].
	\end{align*}

	Finally, by applying Harnack inequality for $\hat{R}_{\widetilde{K}_{\delta}}^1(B_{2r})$ on $\partial B_{3r/2}$, we have
	\begin{align*}
	\varepsilon+\mathrm{cap}_F(L, B_{2r}) \geq \mathrm{cap}_F(\widetilde{K}_{\delta}, B_{2r}) \geq c_2\,(b-a)^{\alpha^{\ast}} \Big[\Big(\frac{3}{2}-a\Big)^{-\alpha^{\ast}}-\Big(2-b\Big)^{-\alpha^{-\ast}}\Big]>0.
	\end{align*}
	Since $\varepsilon>0$ is arbitrary, we finish the proof.
\end{proof}

The idea of the previous lemma can be modified to derive the `spherical symmetrization' result:
\begin{lemma}[Spherical symmetrization]\label{symm}
	Suppose that $F$ is concave and $\alpha^{\ast}(F)<1$. Let $K$ be a compact subset in $B_{r}(x_0)$ such that $K$ meets $S(t):=\{x \in \mathbb{R}^n : |x-x_0|=t\}$ for all $t \in (ar, br)$, where $0<a<b<1$ are constants satisfying $b<\frac{1}{4}$. Then there exists a constant $c=c(n, F, a, b)$ such that
	\begin{align*}
	\mathrm{cap}_{F}(K, B_{2r}) \geq c(n, F, a, b)>0.
	\end{align*}
\end{lemma}

\begin{proof}
	The proof is similar to the one of Lemma \ref{haus2}. By assumption, we can choose $x_{(t)} \in K \cap S(t)$ for all $t \in (ar, br)$. In particular, for small $\delta>0$, we define $x_i:=x_{(ar+2\delta i)}$ for $i=1,2, \cdots, N(\delta)$ so that 
	\begin{align*}
	ar+2\delta \cdot N(\delta) < br \leq ar+2\delta \cdot(N(\delta)+1).
	\end{align*} 
	Note that $N(\delta) \sim \frac{(b-a)r}{\delta}$. Moreover, for $\delta>0$, we define a set $K_{\delta}$ by
	\begin{align*}
		K_{\delta}=\bigcup_{i=1}^{N(\delta)} \overline{B_i},
	\end{align*}
	where $B_i=B_{x_i}(\delta)$. Again recalling Lemma \ref{capacitable} and its proof, for any $\varepsilon>0$, there exists a sufficiently small $\delta>0$ such that
	\begin{align*}
		\mathrm{cap}_F(K_{\delta}, B_{2r}) \leq \mathrm{cap}_F(K, B_{2r})+\varepsilon.
	\end{align*}
		On the other hand, for simplicity, we suppose that the homogeneous solution $V$ is given by
	\begin{align*}
	V(x)=|x|^{-\alpha^{\ast}}
	\end{align*}
	and $\alpha^{\ast}(F) \in (0,1)$. For each $i=1,2, \cdots, N(\delta)$, write
	\begin{align*}
	W_i(x):=\Big(\frac{|x-x_i|}{\delta}\Big)^{-\alpha^{\ast}} \quad \text{and} \quad W(x)=\sum_{i=1}^{N(\delta)}W_i(x).
	\end{align*}
	Since $F$ is concave, $W$ is $F$-subharmonic in $\mathbb{R}^n \setminus \bigcup_{i=1}^{N(\delta)}\{x_i\}$.
	\begin{enumerate}[(i)]
		\item (On $\partial {K}_{\delta}$) For $y \in \partial {K}_{\delta}$, let $y \in \partial B_i$ for some $i$. Then for $j \neq i$, we have
		\begin{align*}
		|y-x_j|\geq |x_i-x_j| -|y-x_i|=|x_i-x_j|-\delta\geq 2|i-j|\delta-\delta,
		\end{align*}
		and so
		\begin{align*}
		W(y) \leq 2(1+2^{-\alpha^{\ast}}+\cdots+N(\delta)^{-\alpha^{\ast}}) \leq 2\Big(1+\int_2^{N(\delta)}\frac{1}{s^{\alpha^{\ast}}}\,\mathrm{d}s\Big) \leq c\, N(\delta)^{1-\alpha^{\ast}}.
		\end{align*}
		Here we used the condition $\alpha^{\ast}<1$.
		
		\item (On $\partial B_{2r}$) For $z \in \partial B_{2r}$, 
		\begin{align*}
		|z-x_i| \geq 2r-br=(2-b)r,
		\end{align*}
		and so
		\begin{align*}
		W(z) \leq \Big(\frac{(2-b)r}{\delta}\Big)^{-\alpha^{\ast}} \cdot N(\delta).
		\end{align*}
		
	\end{enumerate}
	Therefore, for 
	\begin{align*}
	\widetilde{W}(x):=\frac{W(x)-\Big(\frac{(2-b)r}{\delta}\Big)^{-\alpha^{\ast}} \cdot N(\delta)}{c\, N(\delta)^{1-\alpha^{\ast}}},
	\end{align*}
	we have 
	\begin{align*}
	\text{$\widetilde{W}$ is $F$-subharmonic in $B \setminus {K}_{\delta}$, \quad $\widetilde{W} \leq 0$ on $\partial B_{2r}$,\quad and \quad $\widetilde{W} \leq 1$ on $\partial {K}_{\delta}.$}
	\end{align*}
	Note that since ${K}_{\delta}$ and $B_{2r}$ is regular domains, the capacity potential $\hat{R}_{{K}_{\delta}}^1(B_{2r})$ satisfies: 
	\begin{align*}
	\text{$\hat{R}_{{K}_{\delta}}^1(B_{2r}) = 0$ on $\partial B_{2r}$,\quad and \quad $\hat{R}_{{K}_{\delta}}^1(B_{2r})= 1$ on $\partial {K}_{\delta}.$}
	\end{align*} 
	Hence, the comparison principle yields that 
	\begin{align*}
	\hat{R}_{{K}_{\delta}}^1(B_{2r}) \geq \widetilde{W} \quad \text{in $B_{2r} \setminus {K}_{\delta}.$}
	\end{align*}
	In particular, putting $x=x_0+\frac{3}{2}re_1$, we conclude that
	\begin{align*}
	|x-x_i| \leq 3r/2+br=\Big(\frac{3}{2}+b\Big)r,
	\end{align*}
	and so
	\begin{align*}
	\hat{R}_{{K}_{\delta}}^1(B_{2r}) \Big(x_0+\frac{3}{2}re_1\Big) \geq \widetilde{W}\Big(x_0+\frac{3}{2}re_1\Big) &\geq \frac{\Big[\Big(\frac{(3/2+b)r}{\delta}\Big)^{-\alpha^{\ast}}-\Big(\frac{(2-b)r}{\delta}\Big)^{-\alpha^{\ast}}\Big] \cdot N(\delta)}{c\, N(\delta)^{1-\alpha^{\ast}}}\\
	& \geq c_1 (b-a)^{\alpha^{\ast}} \Big[\Big(\frac{3}{2}+b\Big)^{-\alpha^{\ast}}-\Big(2-b\Big)^{-\alpha^{-\ast}}\Big].
	\end{align*}
	Hence,
	\begin{align*}
	\varepsilon+\mathrm{cap}_F(K, B_{2r}) \geq \mathrm{cap}_F({K}_{\delta}, B_{2r}) \geq c_2\,(b-a)^{\alpha^{\ast}} \Big[\Big(\frac{3}{2}+b\Big)^{-\alpha^{\ast}}-\Big(2-b\Big)^{-\alpha^{\ast}}\Big]>0.
	\end{align*}
	Since $\varepsilon>0$ is arbitrary, we finish the proof.
\end{proof}

Let $E$ be a regular set in a ball $B_{2r}$. Let $u=\hat{R}^1_{E}(B_{2r})$ be the capacity potential. For $\gamma \in (0, 1)$, let
\begin{align*}
A_{\gamma}=\{x \in B_{2r} : u(x) < \gamma \}.
\end{align*} 

\begin{lemma}\label{linecap}
	Suppose that $F$ is concave and $\alpha^{\ast}(F)<1$. Then, there exists a constant $c_1>0$ depending only on $n, \lambda, \Lambda$ such that: if 
	\begin{align*}
	\gamma \geq c_1 \mathrm{cap}_F (E, B_{2r}),
	\end{align*}
	then the set $A_{\gamma}$ contains a sphere $S(t):=\{x \in \mathbb{R}^n: |x-x_0|=t\}$ for some $t \in (r/10, r/5)$.
\end{lemma}

\begin{proof}
	For $0<\gamma<1$, let $E_{\gamma}:=\{x \in B_{2r}:u(x) \geq \gamma\}.$ We argue by contradiction: suppose that $A_{\gamma}$ does not contain any $S(t)$ for $t \in (r/10, r/5)$. Then the set $E_{\gamma}$ meets $S(t)$ for all $t \in (r/10, r/5)$ and  we have
	\begin{align*}
		\mathrm{cap}_F(E_{\gamma}, B_{2r}) \geq c(n, F)>0,
	\end{align*}
	by employing Lemma \ref{symm} for $a=1/10$ and $b=1/5$. 
	
	On the other hand, by Lemma \ref{est0}, we have
	\begin{align*}
	\mathrm{cap}_F(E_{\gamma}, B_{2r})=\frac{1}{\gamma}\mathrm{cap}_F(E, B_{2r}).
	\end{align*}
	Combining two estimates above, we obtain
	\begin{align*}
	\gamma \leq \frac{1}{c(n, F)} \, \mathrm{cap}_F(E, B_{2r}).
	\end{align*}
	Therefore, by choosing $c_1=\frac{1}{c(n, F)}+1$, we arrive at a contradiction.
\end{proof}

Now we are ready to prove the necessity of the Wiener criterion, Theorem \ref{wienernec}.
\begin{proof}[Proof of Theorem \ref{wienernec}]
	For simplicity, we write $B_r=B_r(x_0)$.
	Suppose that $\Omega^c$ is not $F$-thick at $x_0 \in \partial \Omega$, i.e.
	\begin{align*}
	\int_0^{1} {\mathrm{cap}_{{F}} (\overline{B_t} \setminus \Omega, B_{2t})}\frac{\mathrm{d}t}{t}<\infty.
	\end{align*}	
	For $\varepsilon>0$ to be determined, choose $r_1>0$ small enough so that 
	\begin{align*}
		\int_0^{r_1}{\mathrm{cap}_{{F}} (\overline{B_t} \setminus \Omega, B_{2t})}\frac{\mathrm{d}t}{t}<\varepsilon.
	\end{align*}
	Set $r_{i+1}=r_i/2$ and 
	\begin{align*}
	a_i={\mathrm{cap}_{{F}} (\overline{B_{r_i}} \setminus \Omega, B_{2r_i})}.
	\end{align*}
	Applying Lemma \ref{capest}, 
	\begin{align*}
		\sum_{i=2}^{\infty}a_i \leq c_0(n, \lambda, \Lambda)\,\varepsilon.
	\end{align*}
	Next, by Lemma \ref{exhaust} and Lemma \ref{capacitable}, for each $i$, choose a regular domain $E_i$ such that $\overline{B_{r_i}} \setminus \Omega \subset E_i$ and
	\begin{align*}
	b_i:={\mathrm{cap}_{{F}} (E_i, B_{2r_i})} <a_i+\varepsilon \cdot 2^{-i}.
	\end{align*}
	Then we have
	\begin{align*}
	\sum_{i=2}^{\infty}b_i \leq (c_0+1) \, \varepsilon
	\end{align*}
	and so $b_i \leq (c_0+1) \, \varepsilon$ for $i=2,3,\cdots.$	
	Moreover, let $u_i:=\hat{R}^1_{E_i}(B_{2r_i})$ be the capacity potential. By Lemma \ref{linecap}, for $\gamma_i=c_1 \cdot b_i$, the set 
	\begin{align*}
	A_i=\{x \in B_{2r_i} :  u_i(x) <\gamma_i\}
	\end{align*}
	contains $S(t_i)$ for some $t_i \in (r_i/10, r_i/5)$. Now by selecting $\varepsilon=\frac{1}{2(c_0+1)c_1}>0$, we have $\gamma_i<1$. In particular, since $u_2=1$ on $E_2$ and $S(t_2) \subset A_2$, we conclude that $S(t_2) \subset \Omega$.
	
	Next, let $f \in C(\partial \Omega)$ be the boundary function defined by
	\begin{align*}
	f(x)=
	\left\{ \begin{array}{ll} 
	1 & \text{if $x \in B_{t_2} \cap \partial \Omega$},\\
	0 & \text{if $x \in \partial \Omega \setminus B_{t_2}$}.
	\end{array} \right.
	\end{align*}
	Then we have the following results for the lower Perron solution $\underline{H}_f=\underline{H}_f(\Omega)$:
	\begin{enumerate}[(i)]
		\item $\underline{H}_f \not\equiv 1$: Choose $r>0$ large enough so that $\Omega \subset B_r$. Moreover, set a domain $\Omega_0:=B_{r} \setminus (B_{t_2} \cap \Omega)$ and a boundary function $f_0 \in C(\partial \Omega_0)$ by
		\begin{align*}
		f_0(x)=
		\left\{ \begin{array}{ll} 
		1 & \text{if $x \in B_{t_2} \cap \partial \Omega$},\\
		0 & \text{if $x \in \partial B_r$}.
		\end{array} \right.
		\end{align*} 
		Then since $B_r$ is regular, we have $\overline{H}_{f_0}(\Omega_0)<1$ in $B_r \setminus B_{t_2}$. On the other hand, for any $v \in \mathcal{L}_f(\Omega)$ and $w \in \mathcal{U}_{f_0}(\Omega_0)$, one can check that $v \leq w$ in $\Omega$ using the comparison principle. Therefore, we conclude that $\underline{H}_f(\Omega) \leq \overline{H}_{f_0}(\Omega_0)$ and so $\underline{H}_f(\Omega) \not\equiv 1$.
		
		\item $\max_{S(t_2)}\underline{H}_f=:M<1$: This is an immediate consequence of the strong maximum principle for $\underline{H}_f$ and part (i).
	\end{enumerate}
	For $u:=\frac{\underline{H}_f-M}{1-M}$ which is $F$-harmonic in $\Omega$ and $u \leq 0$ in $S(t_2)$, we claim that 
	\begin{align} \label{claim}
	\liminf_{\Omega \ni x \to x_0}u(x) < \frac{1}{2}.
	\end{align}
	Indeed, since $S(t_2) \subset B_{r_3}$ and $E_3$ is a regular domain, we have
	\begin{align*}
	u(x) \leq 0 \leq \liminf_{y \to x}u_3(y) \quad \text{for any $x \in \partial B_{t_2}=S(t_2)$}, \\
	\limsup_{y \to x} u(y) \leq 1 = \liminf_{y \to x}u_3(y) \quad \text{for any $x \in \partial E_3$}.
	\end{align*}
	Thus, the comparison principle yields that $u \leq u_3$ in $B_{t_2} \setminus E_3$. In particular, since $S(t_3) \subset A_3$, we observe that 
	\begin{align*}
	u \leq u_3 <\gamma_3 \quad \text{on $S(t_3)$.}
	\end{align*}
	Iterating this argument (for example, consider $u-\gamma_3$ instead of $u$), we conclude that
	\begin{align*}
	u \leq \sum_{k=3}^{i}\gamma_k\leq\sum_{k=3}^{\infty} \gamma_k= c_1 \cdot \sum_{k=3}^{\infty} b_i \leq c_1 \, (c_0+1) \, \varepsilon=\frac{1}{2}	
	 \quad \text{on each $S(t_i)$},
	\end{align*}
	which leads to \eqref{claim}. 
	
	Finally, recalling the definition of $u$, the estimate \eqref{claim} is equivalent to 
	\begin{align*}
	\liminf_{\Omega \ni x \to x_0} \underline{H}_f(x) < 1=f(x_0),
	\end{align*}
	which implies that $x_0 \in \partial \Omega$ is an irregular boundary point.
\end{proof}

\noindent {\bf Acknowledgement:} 
Ki-Ahm Lee was supported by the National Research Foundation of Korea (NRF) grant funded by the Korea government (MSIT): NRF-2021R1A4A1027378.  Se-Chan Lee is supported by Basic Science Research Program through the National Research Foundation of Korea (NRF) funded by the Ministry of Education (2022R1A6A3A01086546).

\end{document}